\documentclass[reqno, centertags, 11pt]{amsart}
\usepackage{geometry, enumerate, enumerate}                
\usepackage{graphicx}
\usepackage[all]{xy}
\usepackage{multicol}
\usepackage{amssymb, color}
\usepackage{epstopdf, soul}
\usepackage{bussproofs}
\usepackage{tcolorbox}
\usepackage{MnSymbol}
\usepackage{blindtext}
\usepackage{hyperref}

\newcommand{\rto}{\blacktriangleright}
\newcommand{\lto}{\blacktriangleleft}

\newcommand{\llnot}{\lnot \lnot}
\newcommand{\C}{\mathrm C}

\newcommand{\mc}[1]{\mathcal{ #1}}

\newcommand{\chc}{\mathtt{CHC}}
\newcommand{\vchc}{\vdash_{\mathtt{CHC}}}
\newcommand{\id}{\RightLabel{(id)}}
\newcommand{\0}{\RightLabel{(0)}}
\newcommand{\one}{\RightLabel{(1)}}
\newcommand{\wel}{\RightLabel{(w-l)}}
\newcommand{\wer}{\RightLabel{(w-r)}}
\newcommand{\cut}{\RightLabel{(cut)}}
\newcommand{\landl}{\RightLabel{($\land$-l)}}
\newcommand{\landr}{\RightLabel{($\land$-r)}}
\newcommand{\lorl}{\RightLabel{($\lor$-l)}}
\newcommand{\lorr}{\RightLabel{($\lor$-r)}}
\newcommand{\implr}{\RightLabel{($\rightarrow$-r)}}
\newcommand{\implla}{\RightLabel{($\rightarrow$-l(a))}}
\newcommand{\impllb}{\RightLabel{($\rightarrow$-l(b))}}
\newcommand{\impllc}{\RightLabel{($\rightarrow$-l(c))}}
\newcommand{\implld}{\RightLabel{($\rightarrow$-l(d))}}
\newcommand{\negl}{\RightLabel{($\lnot$-l)}}
\newcommand{\negr}{\RightLabel{($\lnot$-r)}}
\newcommand{\eqr}{\RightLabel{($=$)}}

\newcommand{\vrp}{\varphi}

\newcommand{\alga}{\mathbf A}
\newcommand{\algb}{\mathbf B}

\newcommand{\lincong}{\theta_{T}}

\providecommand{\U}[1]{\protect\rule{.1in}{.1in}}
\newtheorem{theorem}{Theorem}
\theoremstyle{plain}
\newtheorem{acknowledgement}{Acknowledgement}

\newtheorem{corollary}{Corollary}

\newtheorem{definition}{Definition}
\newtheorem{example}{Example}

\newtheorem{lemma}{Lemma}

\numberwithin{equation}{section}

\title{Intuitionistic logic is a connexive logic}

\author{Davide Fazio}
\address{D. Fazio, Dipartimento di Pedagogia, Psicologia, Filosofia, Universit\`a di Cagliari}
\email{dav.faz@hotmail.it}

\author{Antonio Ledda}
\address{A. Ledda, Dipartimento di Pedagogia, Psicologia, Filosofia, Universit\`a di Cagliari}
\email{antonio.ledda@unica.it}

\author{Francesco Paoli}
\address{F. Paoli, Dipartimento di Pedagogia, Psicologia, Filosofia, Universit\`a di Cagliari}
\email{paoli@unica.it}

\date{\today}                                           
\begin{document}
\maketitle
\begin{abstract}
We show that intuitionistic logic is deductively equivalent to Connexive Heyting Logic ($\mathrm{CHL}$), hereby introduced as an example of a strongly connexive logic with an intuitive semantics. We use the \emph{reverse algebraisation} paradigm: $\mathrm{CHL}$ is presented as the assertional logic of a point regular variety (whose structure theory is examined in detail) that turns out to be term equivalent to the variety of Heyting algebras. We provide Hilbert-style and Gentzen-style proof systems for $\mathrm{CHL}$; moreover, we suggest a possible computational interpretation of its connexive conditional, and we revisit Kapsner's idea of superconnexivity.  
\end{abstract}

\section{Introduction}

Despite being one of the earliest traditions to appear in the development of contemporary nonclassical logics (see e.g. \cite{MP}), \emph{connexive logic} has gone under the radar for quite a while, overshadowed by modal and relevance logics in the debate over entailment and other philosophically driven applications of logic. However, the last two decades have witnessed a spectacular resurgence of interest for this approach \cite{Wanstanford, OW, McCall}. Connexive logics embrace some theses about implication and negation that fail in classical logic, yet are intuitively appealing to many:
\begin{itemize}
\item $\lnot (\varphi \rightarrow \lnot \varphi)  $
(Aristotle 1)

\item $\lnot (\lnot \varphi \rightarrow \varphi) $
(Aristotle 2)

\item $\left( \varphi \rightarrow \psi \right) \rightarrow \lnot \left( 
\varphi \rightarrow \lnot \psi \right)  $
(Boethius 1)

\item $\left( \varphi \rightarrow \lnot \psi \right) \rightarrow \lnot \left( 
\varphi \rightarrow \psi \right) $
(Boethius 2)

\end{itemize}

It is precisely the presence of the above theorems that qualifies a logic as connexive, together with the fact that implications should not always be convertible. The latter \emph{desideratum} is crucial, because connexive logicians are adamant that they are after some notion of implication, not some notion of logical equivalence. These requirements are aptly summarised by Wansing \cite{Wanstanford}, whose definition of a connexive logic we take verbatim (although with some notational changes):

\begin{quotation}
Let $\mathcal{L}$ be a language containing a unary connective $\lnot $
(negation) and a binary connective $\rightarrow $ (implication). A logical
system in a language extending $\mathcal{L}$ is called a connexive logic if
[Aristotle 1, Aristotle 2, Boethius 1 and Boethius 2] are theorems and,
moreover, implication is non-symmetric, i.e., $(\varphi \rightarrow \psi
)\rightarrow (\psi \rightarrow \varphi )$ fails to be a theorem (so that $%
\rightarrow $ can hardly be understood as a bi-conditional). This is the now
standard notion of connexive logic.
\end{quotation}

Some authors like Kapsner \cite{Kapstrong}, have contended that these minimal features are insufficient to meet the intuitive demand of connexivity. Something more is required, namely, that formulas of the form $\varphi \rightarrow \lnot \varphi $ behave like the connexive analogue of contradictions, while formulas of the form $\varphi \rightarrow \psi $ and $\varphi \rightarrow \lnot \psi $ should play the role of connexive contraries. Kapsner says that a connexive logic is\emph{\ strongly connexive} if it satisfies the two additional conditions:

\begin{itemize}
\item In no model, $\varphi \rightarrow \lnot \varphi $ is satisfiable (for
any $\varphi $);

\item In no model, $\varphi \rightarrow \psi $ and $\varphi \rightarrow
\lnot \psi $ are simultaneously satisfiable (for any $\varphi ,\psi $).
\end{itemize}

Strong connexivity, however, is not so easy to implement in practice. Omori and Wansing \cite[p. 382]{OW} observe:

\begin{quotation}
The only strongly connexive logic, to the best of our knowledge, is the
heavily criticized system of Angell-McCall, and it remains to be seen if
there are strongly connexive systems with more intuitive semantics.
\end{quotation}

In this paper, we introduce and investigate a strongly connexive logic -- \emph{Connexive Heyting Logic} -- that is \emph{algebraisable} in the sense of Blok and Pigozzi, hence it occupies the highest rank in the so-called Leibniz hierarchy in abstract algebraic logic \cite{Font}. More to the point, we follow the \emph{reverse algebraisation} approach, as advocated e.g. in \cite{Czelakowski, BR+}: We introduce a suitable class of algebras -- indeed, a subvariety of Sankappanavar's \emph{semi-Heyting algebras} \cite{Sanka} -- whose properties have a pronounced connexive flavour, encoding into its axioms enough deductive power to force the algebraisability of its assertional logic. It is precisely by studying the algebraic properties of such \emph{connexive Heyting algebras}, and in particular the presence of a quaternary deductive term (see below), that we conjectured that they might have been term equivalent, as a variety, to Heyting algebras, and then that Connexive Heyting Logic might have turned out to be deductively equivalent to intuitionistic logic. Both conjectures were indeed correct, as shown below.

The strategy we followed in establishing these results is evidently indebted to a fundamental paper by Spinks and Veroff \cite{SV1, SV2}, who prove that Nelson's constructive logic with strong negation is deductively equivalent to a certain substructural logic. These similarities are even alluded to in the title of the present work.

Let us now summarise the discourse of the paper. In Section \ref{prelim} we rehearse a few preliminary notions of abstract algebraic logic and universal algebra needed in the sequel. In Section \ref{Brouwethius}, which is the core of this work, we introduce the variety of connexive Heyting algebras, study its properties, and show that they are term equivalent to Heyting algebras. In Section \ref{ciaccaelle} we take advantage of this result to establish a deductive equivalence between Connexive Heyting Logic (the assertional logic of connexive Heyting algebras) and intuitionistic logic. Putting to good use such an equivalence, we provide Hilbert-style and Gentzen-style calculi for Connexive Heyting Logic. Some philosophical considerations on the computational meaning of connexive implication and on Kapsner's notion of superconnexivity are reserved for Section \ref{philup}. We conclude in Section \ref{openprob}.

\section{Preliminaries}\label{prelim}

We assume a basic knowledge of universal algebra and abstract algebraic logic on the part of the reader, who is referred to \cite{Burris} and to \cite{Font}, respectively, for any unexplained concept or symbol. None the less, a few notions and results that are important for what follows and may be relatively unfamiliar to the intended readership of this paper are recapitulated in this section. We also assume some familiarity with the fundamental notions of connexive logic, for which the reader can consult \cite{Wanstanford}.

\subsection{Equivalence of logics}

It is nowadays customary to view (propositional) logics as ordered pairs of the form $\mathrm{L} = \langle \mathbf{Fm}_{\mathcal{L}}, \vdash_{\mathrm{L}} \rangle$, where $\mathbf{Fm}_{\mathcal{L}}$ is the formula algebra of some propositional language $\mathcal{L}$ and $\vdash_{\mathrm{L}} \subseteq\wp(Fm_{\mathcal{L}})\times Fm_{\mathcal{L}%
}$ is a binary relation obeying the following conditions for all $\Gamma,\Delta \subseteq Fm_{\mathcal{L}}$, $\varphi, \psi \in Fm_{\mathcal{L}}$, and $\sigma$ an $\mathcal{L}$-substitution (an endomorphism of $\mathbf{Fm}_{\mathcal{L}}$):

\begin{itemize}
\item $\Gamma \vdash_{\mathrm{L}} \varphi$ whenever $\varphi\in \Gamma$. \hfill(\emph{Reflexivity})

\item If\/ $\Gamma \vdash_{\mathrm{L}} \varphi$ and $\Gamma\subseteq \Delta$, then $\Delta \vdash_{\mathrm{L}} \varphi$.
\hfill(\emph{Monotonicity})

\item If $\Delta \vdash_{\mathrm{L}} \varphi$ and $\Gamma \vdash_{\mathrm{L}} \psi$ for every $\psi\in \Delta$, then
$\Gamma \vdash_{\mathrm{L}} \varphi$. \hfill(\emph{Cut})

\item If $\Gamma \vdash_{\mathrm{L}}
\varphi$, then $\sigma(\Gamma) \vdash_{\mathrm{L}} \sigma(\varphi)$. \hfill(\emph{Substitution-invariance})

\end{itemize}

The first three demands in the previous list define the general concept of a \emph{consequence relation}; according to the fourth condition, whether a sentence \emph{logically}
follows from a set of sentences should not depend on the subject matter of the
sentences under consideration, but merely on their logical form. However, one might further contend that the definition of logical consequence should not be tied to any pre-determined type of syntactic unit. In other words, we should make room for consequence relations among sequents or
equations, alongside the traditional ones among formulas, and devise at the same time a
notion of \emph{equivalence} according to which relations on different
syntactic units can be taken to represent the same logic. 

In a 2006 paper \cite{BJ}, Wim Blok and Bjarni J\'{o}nsson take a decisive
step. They suggest to replace the formula algebra in the traditional definition of a consequence relation by an
arbitrary set:

\begin{definition}\label{cioccoblocco}
An \emph{abstract consequence
relation} on a set $A$ is a relation ${\vdash}\subseteq
{\wp(A)\times A}$ obeying the following conditions for all $\Gamma,\Delta \subseteq A$
and for all $a\in A$:

\begin{itemize}
\item $\Gamma \vdash a$ whenever $a\in \Gamma$. \hfill(\emph{Reflexivity})

\item If\/ $\Gamma \vdash a$ and $\Gamma \subseteq \Delta$, then $\Delta \vdash a$. \hfill
(\emph{Monotonicity})

\item If\/ $\Delta \vdash a$ and $\Gamma \vdash b$ for every $b\in \Delta$, then $\Gamma \vdash a$.
\hfill(\emph{Cut})
\end{itemize}

\end{definition}

The absence of an analogue of substitution-invariance in this definition is not that surprising. After all, in a set there is no structure to be preserved, and
consequently no applicable concept of an endomorphism. Blok and J\'{o}nsson's
valuable insight is the observation that the application of substitutions to
propositional formulas (or, for that matter, equations or sequents) behaves
like a multiplication by a scalar. In fact, if $\mathcal{L}$ is a language,
$\varphi$ is an $\mathcal{L}$-formula and $\sigma_{1},\sigma_{2}$ are
$\mathcal{L}$-substitutions, then $\left(  \sigma_{1}\circ\sigma_{2}\right)
\left(  \varphi\right)  =\sigma_{1}\left(  \sigma_{2}\left(  \varphi\right)
\right)  $, and if $\iota$ is the identity $\mathcal{L}$-substitution,
$\iota\left(  \varphi\right)  =\varphi$. Generalising this example, we are led
to the following abstract counterpart of a substitution-invariant consequence relation on formulas. 
\begin{definition}
\begin{itemize}
    \item Let $A$ be a set and $\mathbf{M}=\left\langle M,\cdot,1\right\rangle $ a monoid. $A$ is a \emph{left }$\mathbf{M}$\emph{-act} if there is a
map $\star:M\times A\rightarrow A$ s.t.\ for all $m_{1},m_{2}\in M$ and all
$a\in A$, $\left(  m_{1}\cdot m_{2}\right)  \star a=m_{1}\star\left(
m_{2}\star a\right)  $ and $1\star a=a$.
    \item An abstract consequence relation $\vdash$ on a left $\mathbf{M}$-act $A$
is \emph{action-invariant} if for all $X\cup\left\{  a\right\}
\subseteq A$ and $m\in M$, whenever $X\vdash a$ we also have that $\left\{
m\star x:x\in X\right\}  \vdash m\star a$.
\end{itemize}

\end{definition}

Next, Blok and J\'{o}nsson define a notion of equivalence between
action-invariant abstract consequence relations:

\begin{definition}
Let $\mathbf{M}$ be a monoid, and let $\vdash_{1}$ and $\vdash_{2}$ be two action-invariant abstract consequence relations on the left $\mathbf{M}$-acts $A_{1}$ and $A_{2}$, respectively. $\vdash_{1}$ and $\vdash_{2}$ are \emph{equivalent} if there are mappings $\tau\colon A_{1}\rightarrow\wp\left(  A_{2}\right), \rho\colon
A_{2}\rightarrow\wp\left(  A_{1}\right)$
such that for every $\Gamma\cup\left\{  a\right\}  \subseteq A_{1}$, every
$b \in A_{2}$ and every $m \in M$:
\begin{itemize}
    \item $\Gamma\vdash_{1}a$ iff $\tau\left(  \Gamma\right)  \vdash_{2}
\tau\left(  a\right)  $;
    \item $b\dashv\vdash_{2}\tau\left(  \rho\left(
b\right)  \right)  $;
    \item $\tau (m \star a) = m \star \tau (a) $ and $\rho (m \star b) = m \star \rho (b) $.
\end{itemize}

\end{definition}

Two special cases of this definition are worth flagging. The former is the celebrated notion of \emph{algebraisability} of a logic, due to Blok and Pigozzi \cite{BlokP}. The latter is the notion of \emph{Gentzen algebraisability} of a sequent calculus, which presents different incarnations in the algebraic logic literature; the one we use in this paper is, essentially, to be found in the work of James Raftery \cite{Raftery}.

\begin{definition}\label{algebru}
\begin{itemize}
    \item Let $\mathrm{L} = \langle \mathbf{Fm}_{\mathcal{L}}, \vdash_{\mathrm{L}} \rangle$ be a logic of language $\mathcal{L}$, and $\mathcal{K}$ be a class of similar algebras of the same language. $\mathrm{L}$ is \emph{algebraisable} with equivalent algebraic semantics $\mathcal{K}$ if $\vdash_{\mathrm{L}}$ and the equational consequence relation $\vdash_{\mathcal{K}}$ are equivalent as abstract consequence relations. The sets of all equations $\tau(\varphi)$ and of all formulas $\rho(\varphi,\psi)$ are called a \emph{system of defining equations} and a \emph{system of equivalence formulas}, respectively, for $\mathrm{L}$ and $\mathcal{K}$.
    \item Let $\mathtt{C}$ be a sequent calculus of language $\mathcal{L}$, and $\mathcal{K}$ be a class of similar algebras of the same language. $\mathtt{C}$ is \emph{Gentzen algebraisable} with equivalent algebraic semantics $\mathcal{K}$ if the derivability relation $\vdash_{\mathtt{C}}$ of $\mathtt{C}$ and $\vdash_{\mathcal{K}}$ are equivalent as abstract consequence relations.
\end{itemize}
\end{definition}

As general and wide-ranging as it is, this definition does not quite capture another very natural notion of equivalence between logics, which is roughly an analogue of the algebraic relation of term equivalence between varieties (see e.g. \cite{Pynko, Gyuris, Caleiro} for different precisifications of this idea). The next special case, the only one which is needed for our current purposes, is a common instance of all these notions. For the concept of a translation (of which we also need here only a very special case), see e.g. \cite{Carnielli, Humberstone}.

\begin{definition}
Let $\mathrm{L}_{1}=\left\langle \mathbf{Fm}_{\mathcal{L}},\vdash _{\mathrm{L%
}_{1}}\right\rangle $ and $\mathrm{L}_{2}=\left\langle \mathbf{Fm}_{\mathcal{%
L}},\vdash _{\mathrm{L}_{2}}\right\rangle $ be two logics of language $\mathcal{L}$. A (definitional) \emph{translation} of $\mathrm{L}_{1}$ to $\mathrm{L}_{2}$ is a map $\tau$ such that $\tau(x) = x$ for all $x \in Var_{\mathcal{L}}$, and such that for any $n$-ary connective $g$ in $\mathcal{L}$ there is a (not necessarily primitive) connective $g^\tau$ in $\mathcal{L}$ such that for any ${\mathcal{L}}$-formulas $\varphi_1,...,\varphi_n$, $\tau(g(\varphi_1,...,\varphi_n)) = g^{\tau}(\tau(\varphi_1),...,\tau(\varphi_n))$.
\end{definition}

\begin{definition}\label{pincovalenza}
Let $\mathrm{L}_{1}=\left\langle \mathbf{Fm}_{\mathcal{L}},\vdash _{\mathrm{L%
}_{1}}\right\rangle $ and $\mathrm{L}_{2}=\left\langle \mathbf{Fm}_{\mathcal{%
L}},\vdash _{\mathrm{L}_{2}}\right\rangle $ be two logics of language $%
\mathcal{L}$. $\mathrm{L}_{1}$ and $\mathrm{L}_{2}$ are \emph{deductively
equivalent} if there exist two translations $\tau$ (of $\mathrm{L}_{1}$ to $\mathrm{L}_{2}$) and $\rho$ (of $\mathrm{L}_{2}$ to $\mathrm{L}_{1}$) such that for all $\Gamma \cup \left\{ \varphi \right\} \subseteq Fm_{%
\mathcal{L}}$,

\begin{enumerate}
\item $\Gamma \vdash _{\mathrm{L}_{1}}\varphi $ iff $\tau \left( \Gamma
\right) \vdash _{\mathrm{L}_{2}}\tau \left( \varphi \right) $;

\item $\tau \left( \rho \left( \varphi \right) \right) \dashv \vdash _{%
\mathrm{L}_{2}}\varphi $.
\end{enumerate}
\end{definition}

\subsection{Complements of universal algebra}\label{uaaal}

A thriving literature is available in universal algebra on varieties with a good theory of ideals (e.g., groups, rings, Boolean algebras). There is some consensus to the effect that such varieties, in case the language contains at least one constant, coincide with varieties that are both \emph{subtractive} and \emph{point-regular} (see e.g. \cite{GU}). The relevant definitions follow.

\begin{definition}
A variety $\mathcal{V}$, whose language $\mathcal{L}$ includes a constant $1$, is said to be:
\begin{itemize}
    \item $1$\emph{-subtractive} iff for any congruences $\theta ,\varphi $ on any $\mathbf{A} \in \mathcal{V}$, $1/\theta \circ \varphi = 1/\varphi \circ \theta$;
    \item $1$\emph{-regular} iff for any congruences $\theta ,\varphi $ on any $\mathbf{A} \in \mathcal{V}$, $1/\theta =1/\varphi $ implies $\theta =\varphi $;
    \item $1$\emph{-ideal determined} iff it is both $1$-subtractive and $1$-regular.
\end{itemize}
\end{definition}

We say that $\mathcal{V}$, of language $\mathcal{L}$, is subtractive (resp. point regular, ideal determined) if it is $1$-subtractive (resp. $1$-regular, $1$-ideal determined) for some constant $1$ in $\mathcal{L}$. All the above properties are Maltsev properties; moreover, they are crucially related to properties of the \emph{ideals} and of the \emph{assertional logics} of the varieties at issue. The universal algebraic definition of an ideal, as well as the definition of an assertional logic, are given below.

\begin{definition} \label{ideal term}
Let $\mathcal{V}$ be a variety whose language $\mathcal{L}$ includes a constant $1$.
\begin{enumerate}
\item A formula $\varphi \left(  \overrightarrow{x},\overrightarrow{y}\right)  $ of language $\mathcal{L}$ is a $\mathcal{V}$\emph{-ideal formula} in $\overrightarrow{x}$ iff $\vdash_{\mathcal{V}} \varphi \left(1,...,1,\overrightarrow{y}\right)  \approx 1$.

\item A nonempty subset $J$ of the universe of an $\mathbf{A}\in \mathcal{V}$ is
a $\mathcal{V}$\emph{-ideal} of $\mathbf{A}$ (w.r.t. $1$) iff for any
$\mathcal{V}$-ideal formula $\varphi \left(  \overrightarrow{x},\overrightarrow{y}\right)  $ in $\overrightarrow{x}$ we have that $\varphi^{\mathbf{A}}\left(
\overrightarrow{a},\overrightarrow{b}\right)  \in J$ whenever
$\overrightarrow{a}\in J$ and $\overrightarrow{b}\in A$.
\end{enumerate}
\end{definition}

\begin{definition}
Let $\mathcal{V}$ be a variety whose language $\mathcal{L}$ includes a
constant $1$. The $1$\emph{-assertional logic }of $\mathcal{V}$ is the logic 
$\mathrm{L}_{\mathcal{V}}=\left\langle \mathbf{Fm}_{\mathcal{L}},\vdash _{%
\mathrm{L}_{\mathcal{V}}}\right\rangle $, where%
\begin{equation*}
\Gamma \vdash _{\mathrm{L}_{\mathcal{V}}}\varphi \text{ iff }\left\{ \psi
\approx 1:\psi \in \Gamma \right\} \vdash _{\mathcal{V}}\varphi \approx 1%
\text{.}
\end{equation*}
\end{definition}

For $1$-subtractive varieties, we have the following result \cite{GU, OSV3}:

\begin{theorem}
\label{manzoni}For $\mathcal{V}$ a variety of algebras whose language $%
\mathcal{L}$ includes a constant $1$, the following are equivalent:

\begin{enumerate}
\item $\mathcal{V}$ is $1$-subtractive;

\item There is a binary formula $\varphi(x,y)$ in $Fm_{\mathcal{L}}$ such that $\vdash _{\mathcal{V}} \varphi(1,x) \approx x$ and $\vdash _{\mathcal{V}} \varphi(x,x) \approx 1$.

\end{enumerate}

\end{theorem}

For 1-regular varieties, we have instead \cite{Czelakowski, OSV3, Fichtner}:

\begin{theorem}
\label{lunare}For $\mathcal{V}$ a variety of algebras whose language $%
\mathcal{L}$ includes a constant $1$, the following are equivalent:

\begin{enumerate}
\item $\mathcal{V}$ is $1$-regular;

\item $\mathrm{L}_{\mathcal{V}}$ is algebraisable
with $\mathcal{V}$ as equivalent algebraic semantics;

\item There are binary formulas $\varphi _{1}(x,y),...,\varphi _{n}(x,y)$ in $Fm_{%
\mathcal{L}}$ such that $x\approx y\dashv \vdash _{%
\mathcal{V}}\varphi _{1}\left( x,y\right) \approx 1, ...,\varphi _{n}\left( x,y\right) \approx 1$.
\end{enumerate}
\end{theorem}

The $1$-assertional logic $\mathrm{L}_{\mathcal{V}}$ of a $1$-regular
variety $\mathcal{V}$ can be effectively axiomatised provided an
axiomatisation of $\mathcal{V}$ and a system of equivalence formulas for
$\mathrm{L}$ and $\mathcal{V}$ are\ both known \cite[Thm. 8.0.9]{BR+}:

\begin{theorem}
\label{foscolo}Let $\mathcal{V}$ be a $1$-regular variety of
language $\mathcal{L}$, and let $\rho\left(  \varphi,\psi\right) $ be a system of equivalence formulas for
$\mathrm{L}_{\mathcal{V}}$ and $\mathcal{V}$. Then $\mathrm{L}_{\mathcal{V}}$ is axiomatised by the
following axioms and rules:

\begin{description}
\item[A1] $\vdash_{\mathrm{L}_{\mathcal{V}}}\rho\left(  \varphi,\varphi\right)  $;

\item[A2] $\varphi,\rho\left(  \varphi,\psi\right)  \vdash_{\mathrm{L}_{\mathcal{V}}}\psi$;

\item[A3] $\rho\left(  \varphi,\psi\right)  \vdash_{\mathrm{L}_{\mathcal{V}}}\rho\left(
\psi,\varphi\right)  $;

\item[A4] For each $k$-ary $\mathcal{L}$-connective $c^{k}$, $\bigcup
\nolimits_{i\leq k}\rho\left(  \varphi_{i},\psi_{i}\right)  \vdash
_{\mathrm{L}_{\mathcal{V}}}\rho\left(  c^{k}\left(  \overrightarrow{\varphi}\right)
,c^{k}\left(  \overrightarrow{\psi}\right)  \right)  $;

\item[A5] $\varphi\dashv\vdash_{\mathrm{L}_{\mathcal{V}}}\rho\left(  
\varphi,1  \right)  $;

\item[A6] For each axiom of $\mathcal{V}$ $\varphi\approx\psi$,
\[
\vdash_{\mathrm{L}_{\mathcal{V}}}\rho\left(  \varphi,\psi\right)  \text{.}%
\]

\end{description}
\end{theorem}

Clearly, ideal determined varieties possess the desiderable features of both subtractive and point-regular varieties. In particular, ideals correspond bijectively to congruences in any member of such \cite{BlokP, BlokR, GU}:

\begin{theorem}
\label{terreno}Let $\mathcal{V}$ be a $1$-ideal determined variety, and let $\mathbf{A} \in \mathcal{V}$. The following lattices are isomorphic:

\begin{enumerate}
\item The lattice of all $\mathcal{V}$-ideals (w.r.t. $1$) of $\mathbf{A}$;

\item the lattice of all deductive filters on $\mathbf{A}$ of the $1$-assertional logic of $\mathcal{V}$;

\item the lattice of congruences of $\mathbf{A}$.
\end{enumerate}
\end{theorem}

Varieties with equationally definable principal congruences (EDPC) were introduced by Fried, Gr\"atzer and Quackenbush \cite{FGQ} and extensively studied in algebraic logic as equivalent algebraic semantics of algebraisable logics with the deduction-detachment theorem \cite{KP, EDPC1, EDPC2, EDPC3, EDPC4}. Among the varieties with EDPC, a prominent role is played by varieties with a \emph{quaternary deductive (QD) term}, the latter being a generalisation of the quaternary discriminator (normal transform) to non-semisimple varieties. We start by recalling the relevant definitions.

\begin{definition}\label{EDPCandstuff}
Let $\mathcal{V}$ be a variety of language $\mathcal{L}$. We say that:

\begin{itemize}
\item $\mathcal{V}$ has \emph{equationally definable principal congruences
(EDPC)} if there exist $\mathcal{L}$-identities $\varphi _{1}\approx \psi
_{1},...,\varphi _{n}\approx \psi _{n}$ in the variables $x,y,z,w$ such that
for any $\mathbf{A}\in \mathcal{V}$ and any $a,b,c,d\in A$,%
\begin{equation*}
\left\langle c,d\right\rangle \in \theta ^{\mathbf{A}}\left( a,b\right) 
\text{ iff }\varphi _{i}^{\mathbf{A}}\left( a,b,c,d\right) =\psi _{i}^{%
\mathbf{A}}\left( a,b,c,d\right)
\end{equation*}

for each $i\leq n$;

\item $\mathcal{V}$ has a \emph{quaternary deductive (QD) term} if there
exist an $\mathcal{L}$-formula $\varphi $ in the variables $x,y,z,w$ such
that for any $\mathbf{A}\in \mathcal{V}$ and any $a,b,c,d\in A$,%
\begin{equation*}
\varphi ^{\mathbf{A}}\left( a,b,c,d\right) =\left\{ 
\begin{array}{l}
c\text{ if }a=b\text{;} \\ 
d\text{ if }\left\langle c,d\right\rangle \in \theta ^{\mathbf{A}}\left(
a,b\right).%
\end{array}%
\right.
\end{equation*}
\end{itemize}
\end{definition}

The results in the next theorem, variously due to \cite{EDPC1, EDPC2, KP, vA}, collect the main properties of varieties with a QD term.

\begin{theorem}\label{marziano}
Let $\mathcal{V}$ be a variety with a QD term. Then:

\begin{enumerate}
\item $\mathcal{V}$ is congruence permutable and has EDPC.

\item The join semilattice of compact congruences of any $\mathbf{A}\in 
\mathcal{V}$ is dually relatively pseudocomplemented, namely, there exists a binary operation $\ast$ such that, for any compact members $\theta,\delta,\gamma$ of $\mathrm{Con}(\mathbf{A})$ one has:
\[\theta\subseteq\delta\lor\gamma\quad\Leftrightarrow\quad\delta\ast\theta\subseteq\gamma.\].
\item If $\mathcal{V}$ is $1$-regular, then every compact congruence of any $%
\mathbf{A}\in \mathcal{V}$ is principal.
\end{enumerate}
\end{theorem}
For future use, we specify that the dual relative pseudocomplement $\ast$ of Theorem \ref{marziano}.(2) is such that
\[
\theta ^{\mathbf{A}} (a,b) \ast \theta ^{\mathbf{A}} (c,d) = \theta ^{\mathbf{A}} (\varphi^{\mathbf{A}} (a,b,c,d),d),
\]
where $\varphi$ is the QD term for $\mathcal{V}$.

By Theorem \ref{terreno}, in an algebra $\mathbf{A}$ belonging to a $1$-ideal determined variety $\mathcal{V}$, $\mathcal{V}$-ideals of $\mathbf{A}$ correspond to congruence classes of $1$, in a way that yields a bijective correspondence between the lattice of such $\mathcal{V}$-ideals and the lattice of congruences of $\mathbf{A}$. In light of Theorem \ref{marziano}.(2)-(3), the join-semilattice of principal $\mathcal{V}$-ideals of $\mathbf{A}$ is dually relatively pseudocomplemented.

\section{Connexive Heyting algebras}\label{Brouwethius}

In this section we introduce the variety of \emph{connexive Heyting algebras}, study its properties, and establish its term equivalence with the variety of Heyting algebras. Recall from our introduction that we aim at finding a strongly connexive logic with an intuitive semantics which is, moreover, algebraisable according to Definition \ref{algebru}. The results collected in Subsection \ref{uaaal} suggest the following \textquotedblleft recipe" for obtaining an algebraisable connexive logic:

\begin{itemize}
\item Consider a language $\mathcal{L}$ containing (at least) a negation $\lnot $ and an implication $\rightarrow $ (primitive or definable) and a constant $1$.

\item Define a variety $\mathcal{V}$ of language $\mathcal{L}$ and make sure that the different versions of Aristotle's and Boethius' laws evaluate at $1$ in each $\mathbf{A}\in \mathcal{V}$.

\item Make sure that the symmetry of implication has a counterexample in some $\mathbf{A}\in \mathcal{V}$.

\item Encode enough properties into $\rightarrow $ so that the set $\left\{\varphi \rightarrow \psi ,\psi \rightarrow \varphi \right\} $ witnesses $1$-regularity for $\mathcal{V}$.

\item Apply Theorem \ref{lunare} to the $1$-assertional logic of $\mathcal{V}$.
\end{itemize}

Our candidate $\mathcal{V}$ is the variety $\mathcal{CHA}$ of connexive Heyting algebras, to be defined below. It is a subvariety of \emph{semi-Heyting algebras}, an important and well-studied class introduced by Sankappanavar in 2007 \cite{Sanka} and investigated e.g. in \cite{DCV, CV}.

\subsection{Definition and elementary properties}

Let $\mathcal{L}_{CH}$ the language $\left\langle 2,2,2,0,0\right\rangle $, whose operation symbols are respectively denoted by $\land$ (meet), $\lor$ (join), $\rightarrow$ (implication), $0$ (falsity) and $1$ (truth). The following identities of language $\mathcal{L}_{CH}$ will be considered in what follows ($\lnot x$ is short for $x\rightarrow 0$):

\begin{description}
\item[C1] $(x\rightarrow y)\rightarrow ((y\rightarrow z)\rightarrow
(x\rightarrow z))\approx 1$;

\item[C2] $(x\rightarrow y)\rightarrow \lnot (x\rightarrow \lnot y))\approx 1$;

\item[C3] $x\wedge (x\rightarrow y)\approx x\wedge y$;

\item[C4] $x\rightarrow y\leq \left( z\wedge x\right) \rightarrow (z\wedge y)$;

\item[C5] $x\rightarrow y\leq \left( z\vee x\right) \rightarrow (z\vee y)$;

\item[C6] $x \land (y \rightarrow z) \approx x \land ((x \land y) \rightarrow (x \land z)$;

\item[C7] $x\rightarrow x \approx 1$.

\end{description}
\begin{definition}\label{putacusu}
A \emph{semi-Heyting algebra} is an algebra $\mathbf{A}=\left\langle
A,\wedge ,\vee ,\rightarrow ,0,1\right\rangle $ of language $\mathcal{L}%
_{CH}$ such that:
\begin{itemize}
    \item $\left\langle A,\wedge ,\vee ,0,1\right\rangle $ is a distributive
lattice with bottom element $0$, top element $1$, and induced order $\leq$;
    \item the identities C3, C6, and C7 hold.
\end{itemize}
 \end{definition}

The next lemma is proved in \cite{Sanka}.

\begin{lemma}\label{lem:connexalgarithm1}
Let $\alga$ be a semi-Heyting algebra. The following hold, for any $a,b\in A$:
\begin{multicols}{2}
\begin{enumerate}
\item $1\rightarrow a = a$;
\item $a\rightarrow b = 1$ implies $a \leq b$;
 \item $a \leq b\to (a\land b)$;
 \item$a\leq\lnot b$ if and only if $a\land b=0$;
 \item $a \leq a\to 1$;
 \item$a\leq (a\to b)\to b$;
\item $a \leq \lnot \lnot a $;
\item $a \land \lnot a = 0$;
\item$a\to 0 \leq 0\to a$;
\item$\neg a=\neg\neg\neg a$.
\end{enumerate}
\end{multicols}
Moreover, the proof of items (1), (2), (3), (5) and (6) does not depend on C6 or C7.
\end{lemma}

We now present the notion that will be at the centre of the present paper.

\begin{definition}\label{putacaso}
A \emph{connexive Heyting algebra} is an algebra $\mathbf{A}=\left\langle
A,\wedge ,\vee ,\rightarrow ,0,1\right\rangle $ of language $\mathcal{L}_{CH}$ such that:

\begin{itemize}
\item $\left\langle A,\wedge ,\vee ,0,1\right\rangle $ is a distributive
lattice with bottom element $0$, top element $1$, and induced order $\leq$;
\item the identities C1, C2, C3, C4, and C5 hold.
 \end{itemize}
 
\end{definition}

Connexive Heyting algebras form a variety, hereafter noted $\mathcal{CHA}$. We show that $\mathcal{CHA}$ is a subvariety of the variety of semi-Heyting algebras.

\begin{lemma}
Every connexive Heyting algebra is a semi-Heyting algebra.
\end{lemma}

\begin{proof}
It suffices to show that C6 and C7 hold in every connexive Heyting algebra. We will be free to use items (1), (2), (5) and (6) in Lemma \ref{lem:connexalgarithm1}, which, as already observed, do not depend on either C6 or C7. Observe, moreover, that if $\alga$ is a connexive Heyting algebra and $d,e,f\in A$, we have that $(d\rightarrow e)\land(e\rightarrow f)\leq d\rightarrow f$. 

Now, let again $\alga$ be a connexive Heyting algebra and $a,b,c\in A$. By C4 $a \land (b \rightarrow c) \leq a \land ((a \land b) \rightarrow (a \land c))$. Conversely,
\[%
\begin{array}
[c]{lll}
   a \land ((a \land b) \rightarrow (a \land c))&= a \land (b \rightarrow(a \land b))\land((a \land b) \rightarrow (a \land c))& \text{Lm. \ref{lem:connexalgarithm1}.(3)-(6)}\\
   &\leq a \land (b \rightarrow(a \land c))& \\
   &= a \land (a \rightarrow 1) \land (b \rightarrow(a \land c))& \text{Lm. \ref{lem:connexalgarithm1}.(5)}\\
   &\leq a \land ((a \land c)\rightarrow c) \land (b \rightarrow(a \land c))& \text{C4}\\
   &\leq a \land (b \rightarrow c). &\\
\end{array}
\]
Hence C6 holds. For C7,  by Lemma \ref{lem:connexalgarithm1}.(1)-(2) and C1, $1=1\rightarrow 1\leq(1\rightarrow a)\rightarrow(1\rightarrow a)=a\rightarrow a$.
\end{proof}

We provide a finite example of a connexive Heyting algebra (called L9 by Sankappanavar \cite[Thm. 4.1]{Sanka} and also mentioned by Kapsner and Omori, see \cite{Kapsomo})) showing both that this class is nonempty and that implication, in general, fails to be symmetric therein. 

\begin{example}\label{treelementi} Consider the $3$-element bounded chain $\mathbf{KO}_{3}= \langle \{0,a,1\},\land,\lor,\rightarrow,0,1 \rangle$ equipped with a binary operation $\rightarrow$ according to the following table:
\begin{center}
        \begin{minipage}[c]{.40\textwidth}
\centering
{\tiny\xymatrix@C=0.3em{
&1\ar@{-}[d]&\\
&a\ar@{-}[d]&\\
&0&}}          
        \end{minipage}%
        \hspace{10mm}%
        \begin{minipage}[c]{.40\textwidth}
\centering
\begin{tabular}{l|lll}
$\rightarrow$ & 1 & a & 0 \\ \hline
1             & 1 & a & 0 \\
a             & 1 & 1 & 0 \\
0             & 0 & 0 & 1
\end{tabular}
\end{minipage}
\end{center}
Note that $\mathbf{KO}_{3}$ is a connexive Heyting algebra. Moreover, it satisfies neither the identity $(x\rightarrow y)\rightarrow(y\rightarrow x)\approx 1$ nor the quasi-identity $$x\rightarrow y \approx 1 \curvearrowright y\rightarrow x \approx 1. $$
Indeed, e.g. $(a \to 1) \to (1 \to a) = 1 \to a = a$.

\end{example}

We now list some elementary arithmetical properties of $\mathcal{CHA}$.

\begin{lemma}\label{lem:arit}
Let $\alga$ be a connexive Heyting algebra. The following hold, for any $a,b,c\in A$:
 \begin{multicols}{2}
\begin{enumerate}
 \item $(a\rightarrow b)\land(b\rightarrow c)\leq a\rightarrow c$;
 \item $(a\rightarrow b)\rightarrow((c\rightarrow a)\rightarrow(c\rightarrow b)) = 1$;
 \item if $a\leq b$ then $\lnot b\leq \lnot a$;
 \item $a \leq a\to 1\leq b\to(a\to b)$;
 \item if $\lnot a=1$, then $a=0$;
 \item $\lnot(a\to \lnot a)=1$:
\item$a\to \lnot a=0=\lnot a\to a$;
\item$(a\to b)\land (a\to\lnot b)=0$;
\item$(a\to b)\to(a\to \lnot b)=0$;
\item$(a\to 1)\to \neg a=0$;
\item$0\to a=(a\to 0)\to 1$;
\item $a \rightarrow \lnot \lnot a = 1 $;
\item$0\to a=a\to 0$;
\item$\neg b=\neg((b \to a) \to a)=\neg a\to~(b\to ~a)$;
\item$(a\to 1)\land \lnot a=0$;
\item$\lnot a=(a\to 1)\to0$;
\item$(a\to b)\to1=\neg(a\to\neg b)$;
\item$\neg\neg a=a\to1$;
\item$\neg(a\to b)=\neg(b\to a)$;
\item$\neg(a\to b )=a\to \neg b$;
\item$a\rightarrow b = 0$ if and only if $a\rightarrow\neg b=1$;
\item $a\rightarrow b = 1$ implies $a\rightarrow\neg b = 0$.
\end{enumerate}
 \end{multicols}
\end{lemma}
\begin{proof}
(1) Clear.\\
(2) Note that Lemma \ref{lem:connexalgarithm1}.(1) and C1 entail that $$(((b\rightarrow d)\rightarrow(c\rightarrow d))\rightarrow a)\rightarrow((c\rightarrow b)\rightarrow a)=1.$$ Upon setting $c:= 1$, again by Lemma \ref{lem:connexalgarithm1}.(1) it follows that $(((b\rightarrow c)\rightarrow c)\rightarrow a)\rightarrow(b\rightarrow a)=1$. Therefore, one obtains
$$((((b\rightarrow d)\rightarrow(c\rightarrow d))\rightarrow(c\rightarrow d))\rightarrow((c\rightarrow b)\rightarrow(c\rightarrow d)))\rightarrow((b\rightarrow d)\rightarrow((c\rightarrow b)\rightarrow(c\rightarrow d))) = 1$$ and also $$((((b\rightarrow d)\rightarrow(c\rightarrow d))\rightarrow(c\rightarrow d))\rightarrow((c\rightarrow b)\rightarrow(c\rightarrow d)))=1.$$Hence, Lemma \ref{lem:connexalgarithm1}.(1) yields the desired conclusion.\\
 (3) If $a\leq b$ then $\lnot b\land a = a \land \lnot b \leq b \land \lnot b = 0$, whence, using Lemma \ref{lem:connexalgarithm1}.(4), $\lnot b \leq \lnot a$.\\
(4) The first inequality is Lemma \ref{lem:connexalgarithm1}.(5). For the second one, $a\to 1\leq (1\to b)\to (a\to b)=b\to(a\to b)$, by C1 and Lemma \ref{lem:connexalgarithm1}.(1)-(2).\\
(5) If $1=\lnot a=a\to 0$, then $a \leq 0$.\\
(6) By C2, C7 and Lemma \ref{lem:connexalgarithm1}.(1).\\
(7) The fact that $a\to \lnot a=0$ follows from items (5) and (6). Moreover, 
\[
\lnot a\to a\leq (\lnot a\to\lnot a)\to (\lnot a\to a)=1\to 0=0.
\]
(8) By C2 and Lemma \ref{lem:connexalgarithm1}.(4).\\
(9) By C2 and items (5) and (8).\\
(10) From (9), setting $b := 1$.\\
(11) By C1,
\[
0\to a\leq (a\to 0)\to (0\to 0) = (a\to 0)\to1.
\]
So, $0\to a\leq(a\to 0)\to 1$. Also, $(a\to 0)\to 1\leq (1\to a)\to (\neg a\to a)=a\to 0 \leq 0 \to a$, by C1, Lemma \ref{lem:connexalgarithm1}.(2)-(9) and item (7).\\
(12)
\[
a \to \lnot \lnot a = (1 \to a) \to ((1 \to (a \to 0)) \to 0) = 1.
\]
(13) By the proof of item (11).\\
(14)
\begin{align*}
\neg b\leq& (0\to a)\to (b\to a)& \text{C1, Lm. \ref{lem:connexalgarithm1}.(2)}\\
\leq& ((b\to a)\to a)\to ((0\to a)\to a)& \text{C1, Lm. \ref{lem:connexalgarithm1}.(2)}\\
=& ((b\to a)\to a)\to ((a\to 0)\to a)& \text{(13)}\\
=&\neg ((b\to a)\to a)& \text{(7)}\\
\leq& \lnot b.& \text{(3), Lm. \ref{lem:connexalgarithm1}.(6)}\\
\end{align*}
(15) $a\to 1\leq \lnot(a \to \lnot 1) = \lnot \lnot a$ by C2 and Lemma \ref{lem:connexalgarithm1}.(2), hence the claim follows from Lemma \ref{lem:connexalgarithm1}.(4).\\
(16) Since $a\leq a\to 1$ by Lemma \ref{lem:connexalgarithm1}.(5), $(a\to 1)\to 0\leq a\to 0$ by item (3). Also, from (15) and Lemma \ref{lem:connexalgarithm1}.(4) we obtain that $a\to0\leq(a\to 1)\to 0$.\\
(17) 
\begin{align*}
1=&(a\to b)\to\neg(a\to\neg b)& \text{C2}\\
\leq&(\neg(a\to\neg b)\to1)\to((a\to b)\to1).& \text{C1, Lm. \ref{lem:connexalgarithm1}.(2)}
\end{align*} 
Then, by Lemma \ref{lem:connexalgarithm1}.(2)-(5) and C1, $\lnot(a \to \lnot b) \leq \neg(a\to\neg b)\to1\leq(a\to b)\to1$. 
Also, 
\begin{align*}
((a\to b)\to1)\land(a\to\neg b)\leq&((1\to(a\to \lnot b))\to ((a\to b)\to (a\to\neg b)))\land (a\to\neg b)& \text{C1, Lm. \ref{lem:connexalgarithm1}.(2)}\\
=&\neg(1\to(a\to \lnot b))\land (a\to \neg b)&\text{(9)}\\
=&\neg(a\to \lnot b)\land (a\to\neg b)& \text{Lm. \ref{lem:connexalgarithm1}.(1)}\\
=&0,&\text{Lm. \ref{lem:connexalgarithm1}.(8)}
\end{align*}
i.e. $(a\to b)\to1\leq\neg(a\to\neg b)$.\\
(18) Set $a := 1$ and $b := a$ in (17).\\
(19) \begin{align*}
(a\to b)\land ((b\to \neg a)\to1)=&(a\to b)\land\neg(b\to \neg\neg a)&\text{(17)}\\
\leq&((b\to\neg\neg a)\to(a\to \neg\neg a))\land\neg(b\to \neg\neg a)& \text{C1, Lm. \ref{lem:connexalgarithm1}.(2)}\\
=&((b\to\neg\neg a)\to1)\land\neg(b\to \neg\neg a)&\text{(12)}\\
=&0.&\text{(15)}
\end{align*}
Therefore,
\begin{align*}
0=&(a\to b)\land((b\to\neg a)\to1)\\
=&(a\to b)\land\neg\neg(b\to\neg a)&\text{(18)}\\
=&(a\to b)\land\neg((b\to a)\to1)&\text{(17)}\\
=&(a\to b)\land\neg(b\to a),&\text{(18), Lm. \ref{lem:connexalgarithm1}.(10)}
\end{align*}
i.e. $\neg(a\to b)\geq\neg(b\to a)$.\\
(20) \begin{align*}
\lnot(a \to b)=&\lnot \lnot \lnot(a \to b)&\text{Lm. \ref{lem:connexalgarithm1}.(10)}\\
=& \lnot((a \to b) \to 1)&\text{(18)}\\
=& \lnot \lnot(a \to \lnot b)& \text{(17)}\\
=& \lnot \lnot(\lnot b \to a)& \text{(19)}\\
\leq& a \to ((\lnot b \to a) \to a) &\text{(4)-(18)}\\
=& a \to \lnot b &\text{(13)}
\end{align*}
The last line is justified as follows. $\lnot b$ can be replaced by $(\lnot b \to a) \to a$ since, by Lemma \ref{lem:connexalgarithm1}.(6)-(7) and (10), as well as item (14), $\lnot b \leq (\lnot b \to a) \to a \leq \lnot \lnot ((\lnot b \to a) \to a) = \lnot \lnot \lnot b = \lnot b$. 
\\
(21) If $a\rightarrow b=0$, then $\neg(a\rightarrow b)=a\rightarrow\neg b=1$ by (20). The converse holds as well by item (8).\\
(22) follows from items (20) and (21).
\end{proof}

By C2, C7 and Lemma \ref{lem:connexalgarithm1}.(1)-(2), the relation $\{ \langle a,b\rangle : a \rightarrow b = 1 \}$ is a partial ordering on any $\mathbf{A} \in \mathcal{CHA}$. By Example \ref{treelementi}, it is generally stronger than the ordering induced by the lattice operations. A more precise characterisation is contained in the following theorem. 

\begin{theorem}\label{thm:ndcd-rdr}
Let $\alga$ be a connexive Heyting algebra. The following hold, for any $a,b \in A$:
\[
a\to b=1 \text{ if and only if }a\leq b\text{ and }\lnot a=\lnot b.
\]
\end{theorem}
\begin{proof}
By Lemmas \ref{lem:connexalgarithm1}.(2) and \ref{lem:arit}.(3), if $a\to b=1$, then $a\leq b$ and also $\lnot b \leq \lnot a$. Moreover, by Lemmas \ref{lem:connexalgarithm1}.(2) and \ref{lem:arit}.(2)-(13), $1 = a\to b \leq (0 \to a) \to (0 \to b) = (a \to 0) \to (b \to 0)$, hence $\lnot a \leq \lnot b$.

Conversely,
\begin{align*}
1=&a \to \llnot a &\text{Lm. \ref{lem:arit}.(12)} \\
\leq& (a\land b) \to (\llnot a\land b)&\text{C4}\\
=&a\to (\llnot b \land b)\\
=&a \to b &\text{Lm. \ref{lem:connexalgarithm1}.(7)}
\end{align*}
\end{proof}

It may be expedient to observe that axiom C1 in the definition of connexive Heyting algebras can be equivalently replaced by the only seemingly weaker condition $x \rightarrow y \leq (y \rightarrow z) \rightarrow (x \rightarrow z)$. Observe first that Theorem \ref{thm:ndcd-rdr} does not depend on the full version of C1, but only on the above-mentioned condition. Also, let $\mathbf{A}$ satisfy all the remaining axioms of $\mathcal{CHA}$, and let $a,b,c \in A$. We have that:
\begin{align*}
\lnot(a \to b) =&a \to \lnot b &\text{Lm. \ref{lem:arit}.(20)}\\
\leq & (c \to a) \to (c \to \lnot b) &\text{Lm. \ref{lem:arit}.(2)}\\
=& (c \to a) \to \lnot (c \to b) &\text{Lm. \ref{lem:arit}.(20)}\\
=& (c \to a) \to \lnot (b \to c) &\text{Lm. \ref{lem:arit}.(19)}\\
=& \lnot ((c \to a) \to (b \to c)) &\text{Lm. \ref{lem:arit}.(20)}\\
=& \lnot ((b \to c) \to (c \to a)) &\text{Lm. \ref{lem:arit}.(19)}\\
=& (b \to c) \to \lnot (c \to a) &\text{Lm. \ref{lem:arit}.(20)}\\
=& (b \to c) \to \lnot (a \to c) &\text{Lm. \ref{lem:arit}.(19)}\\
=& \lnot ((b \to c) \to (a \to c)). &\text{Lm. \ref{lem:arit}.(20)}
\end{align*}
(We notice that none of the results used above depend on the full version of C1 either.) Hence, by Theorem \ref{thm:ndcd-rdr}, $(a\to b)\to((b\to c)\to(a\to c))=1$. 

Also, a stronger version of Boethius' law holds in $\mathcal{CHA}$.

\begin{lemma}\label{lem:C2}
Let $\alga$ be a connexive Heyting algebra. The following hold, for any $a,b \in A$:
\[
(a\to b)\to((b\to c)\to\neg(a\to\neg c))=1.
\]
\end{lemma}
\begin{proof}
By C2, Lemma \ref{lem:connexalgarithm1}.(1) and Lemma \ref{lem:arit}.(2), we have that
\begin{align*}
1 =& ((a\to c)\to\neg(a\to\neg c))\to(((b\to c)\to(a\to c))\to((b\to c)\to\neg(a\to\neg c)))\\
=& 1\to(((b\to c)\to(a\to c))\to((b\to c)\to\neg(a\to\neg c)))\\
=& ((b\to c)\to(a\to c))\to((b\to c)\to\neg(a\to\neg c)).
\end{align*}
Since $(a\to b)\leq(b\to c)\to(a\to c)$, by C1 and Theorem \ref{thm:ndcd-rdr}, we have on the one hand that $a \to b \leq (b\to c)\to\neg(a\to\neg c)$, and on the other that $\neg(a\to b)=\neg((b\to c)\to(a\to c))=\neg((b\to c)\to\neg(a\to\neg c))$, and therefore $(a\to b)\to((b\to c)\to\neg(a\to\neg c))=1$, again by Theorem \ref{thm:ndcd-rdr}.
\end{proof}

\subsection{Structure theory}

After surveying some of the most elementary arithmetical properties of $\mathcal{CHA}$, we now delve into its structure theory, with an eye to establishing some crucial underpinnings of the term equivalence result that follows. For a start, we observe that $\mathcal{CHA}$ is an ideal determined variety.

\begin{lemma}\label{stamazza}
$\mathcal{CHA}$ is a $1$-ideal determined variety.
\end{lemma}

\begin{proof}
By Theorem \ref{manzoni}, C7 and Lemma \ref{lem:connexalgarithm1}.(1), the formula $x \to y$ witnesses $1$-subtractivity for $\mathcal{CHA}$. By Theorem \ref{lunare}, C7 again and Lemma \ref{lem:connexalgarithm1}.(2), the set $\{ x \to y, y \to x\}$ witnesses $1$-regularity for $\mathcal{CHA}$.
\end{proof}

In light of Theorem \ref{terreno}, in any connexive Heyting algebra $\mathbf{A}$ we have a lattice isomorphism between the lattices of congruences of $\mathbf{A}$, of $\mathcal{CHA}$-ideals of $\mathbf{A}$, and of congruence classes of $1$ of some congruence on $\mathbf{A}$. If $\theta$ is such a congruence, its congruence class of $1$ is customarily denoted by $1/\theta$. In what follows, we show that such congruence classes of $1$ are nothing but the \emph{lattice filters} of $\mathbf{A}$. This result is known to hold, more generally, for the variety of semi-Heyting algebras. The proof we reproduce here simplifies to some extent the one in \cite[Thm. 5.4]{Sanka}, thanks to the additional axioms of $\mathcal{CHA}$.

Some notational and terminological explanations are now in order. Let $\alga \in \mathcal{CHA}$. We denote by $\mathrm{Fi}(\alga)$ the lattice of lattice filters of $\alga$, as well as its universe. If $C\subseteq A$,  $\mathrm{Fg}(C)$ will denote the lattice filter generated by $C$, i.e. the smallest filter of $\alga$ containing $C$. It is well known that, for any $C\subseteq A$, one has
\[\mathrm{Fg}(C)=\{y\in A:y\geq x_{1}\land\dots\land x_{n},\, x_{1},\dots,x_{n}\in C,\, n\geq 1\}.\]
Given $\alga \in \mathcal{CHA}$ and $F\subseteq A$, we also set $$\Theta(F):=\{ \langle x,y \rangle \in A^{2}: x\rightarrow y,y\rightarrow x\in F\}.$$ 
\begin{lemma}\label{lem: filterscongr}
Let $\alga$ be a connexive Heyting algebra. The following hold:
\begin{enumerate}
\item  For any $\theta\in\mathrm{Con}(\alga)$, $1/\theta \in \mathrm{Fi}(\alga)$.
\item For any $F \in \mathrm{Fi}(\alga)$, $\Theta(F)\in\mathrm{Con}(\alga)$. 
\item $1/\Theta(F)=F$ and $\theta = \Theta(1/\theta)$.   
\end{enumerate}
\end{lemma}
\begin{proof}
$(1)$ is straightforward. Concerning $(2)$, let $F$ be a lattice filter over $\alga$. Reflexivity and transitivity of  $\theta(F)$ follow by C7, Lemma \ref{lem:arit}.(1) and Lemma \ref{lem:connexalgarithm1}.(1). Symmetry holds trivially. Now, if $a\rightarrow b\in F$, then $a\land c\rightarrow b\land c, a\lor c\rightarrow b\lor c, (c\rightarrow a)\rightarrow(c\rightarrow b),(b\rightarrow c)\rightarrow (a\rightarrow c)\in F$ follow by applying C1, C4, C5 and Lemma \ref{lem:arit}.(2). Therefore $\Theta({F})$ is compatible with the operations and so it is a congruence. For (3), if $a\in F$, then one has $1\rightarrow a,a\rightarrow 1\in F$, by Lemma \ref{lem:connexalgarithm1}.(1) and Lemma \ref{lem:arit}.(4). Hence $a\in 1/\Theta(F)$. Conversely, $a\in[1]_{\Theta(F)}$ entails that $1\rightarrow a = a\in F$. Lastly, by C7, if $a \theta b$ then $a \to b \theta b \to b = 1 = a \to a \theta b \to a$, while if $a \to b \theta 1 \theta b \to a$, then in $\alga / \theta$ we have that $a \to b = 1 = b \to a$, and thus $a=b$, whence in $\alga$ we obtain $a \theta b$.
\end{proof}

Next, we show that $\mathcal{CHA}$ has a QD term, which, as we have observed in Subsection \ref{uaaal}, is a bountiful property in terms of implying many other desirable features for a variety. 

\begin{theorem}\label{genoveffa}
$\mathcal{CHA}$ has a QD term.
\end{theorem}

\begin{proof}
Recall that, in light of Definition \ref{EDPCandstuff}, we have to find an $\mathcal{L}_{CH}$-formula $\varphi $ in the variables $x,y,z,w$ such
that for any $\mathbf{A}\in \mathcal{CHA}$ and any $a,b,c,d\in A$,%
\begin{equation*}
\varphi ^{\mathbf{A}}\left( a,b,c,d\right) =\left\{ 
\begin{array}{l}
c\text{ if }a=b\text{;} \\ 
d\text{ if }\left\langle c,d\right\rangle \in \theta ^{\mathbf{A}}\left(
a,b\right).%
\end{array}%
\right.
\end{equation*}
Let $\psi(x,y,z) = ((z\rightarrow y)\rightarrow x)\land ((x\rightarrow y)\rightarrow z)$ and $\chi(x,y,z):= (x\leftrightarrow y)\land z$, where $x\leftrightarrow y:=(x\rightarrow y)\land(y\rightarrow x)$. We set $\varphi(x,y,z,w):=\psi(\chi(x,y,z),\chi(x,y,w),w)$. Explicitly: 
\[\varphi(x,y,z,w) = ((w\rightarrow((x\leftrightarrow y)\land w))\rightarrow((x\leftrightarrow y)\land z))\land((((x\leftrightarrow y)\land z)\rightarrow((x\leftrightarrow y)\land w))\rightarrow w).\]
A straightforward computation, involving C7 and Lemma \ref{lem:connexalgarithm1}.(6), shows that if $a=b$ then $\varphi ^{\mathbf{A}}\left( a,b,c,d\right) = c$. If $\langle c,d \rangle \in \theta^{\mathbf{A}}(a,b)$, then by Lemma \ref{lem: filterscongr} we have that $c \leftrightarrow d \in \mathrm{Fg}(a \leftrightarrow b)$, i.e., $a \leftrightarrow b \leq c \leftrightarrow d$. Thus $(a \leftrightarrow b) \land c \leq (a \leftrightarrow b)\land c \land (c \leftrightarrow d) = (a \leftrightarrow b)\land c \land d \land (d \to c) \leq (a \leftrightarrow b) \land d$. Similarly $(a \leftrightarrow b) \land d \leq (a \leftrightarrow b) \land c$, and hence $(a \leftrightarrow b) \land c = (a \leftrightarrow b) \land d$. It readily follows, using C7 and Lemma \ref{lem:connexalgarithm1}.(6), that $\varphi ^{\mathbf{A}}\left( a,b,c,d\right) = d$.
\end{proof}

The following corollary to the foregoing theorem had already been established by Sankappanavar for semi-Heyting algebras \cite[Cor. 5.7, Thm. 5.8]{Sanka}.

\begin{corollary}
$\mathcal{CHA}$ has EDPC and is congruence permutable.
\end{corollary}

\begin{proof}
By Theorem \ref{marziano}. Observe that the formula that witnesses congruence permutability is none other than $\psi(x,y,z)$ in the previous theorem, while the single identity that witnesses EDPC is (still retaining the conventions from the previous theorem) $m(x,y,z) \approx m(x,y,u)$. 
\end{proof}

By Theorem \ref{marziano}.(2), the join semilattice of compact congruences of any $\alga \in \mathcal{V}$ is dually relatively pseudocomplemented. However, we can say much more in the present case. On the one hand, by Lemma \ref{stamazza} and Lemma \ref{lem: filterscongr}, congruences on $\alga$ bijectively correspond to lattice filters of $\alga$. On the other hand, by Theorem \ref{marziano}.(3) and Lemma \ref{stamazza} again, the join semilattice of principal lattice filters of $\alga$ must be dually relatively pseudocomplemented as well. By the remarks following Theorem \ref{marziano}, we can actually \emph{compute} such dual relative pseudocomplements by first determining the behaviour of the former on principal congruences of the form $\theta(x,1)$, and then cashing out the behaviour of latter on their $1$-classes, i.e., on principal filters. Thus, we have, using Lemma \ref{lem:connexalgarithm1}.(7) and Lemma \ref{lem:arit}.(18)-(19):
\begin{align*}
\theta (a,1) \ast \theta (b,1) =& \theta (\varphi^{\mathbf{A}} (a,1,b,1),1) \\
=& \theta ((a \to (a \land b)) \land (((a \land b) \to a) \to 1),1)\\
=& \theta ((a \to (a \land b)) \land \lnot \lnot ((a \land b) \to a),1)\\
=& \theta ((a \to (a \land b)) \land \lnot \lnot (a \to (a \land b)),1)\\
=& \theta (a \to (a \land b),1).
\end{align*}
Hence, the principal filter $Fg(a) \ast Fg(b)$ is generated by $a \to (a \land b)$. Since there is a \emph{dual} order isomorphism between the poset reduct of $\alga$ and the poset of lattice filters of $\alga$, one is somehow led to surmise that the element $a \to (a \land b)$ must have some features that make it akin to a relative pseudocomplement, i.e., to a Heyting implication. This was the main insight that made us conjecture, and then prove, the results in the next subsection, even before we got acquainted with Sankappanavar's results on semi-Heyting algebras. 

\subsection{Term equivalence with Heyting algebras} 

Our next goal is to show that any connexive Heyting algebra has a term reduct that is a Heyting algebra. This is a property that holds, more generally, for all semi-Heyting algebras \cite[Lm. 4.1]{DCV}. However, capitalising on the stronger structure results obtained so far for $\mathcal{CHA}$, we can give an essentially different proof of the same theorem. Hereafter, whenever $\alga\in\mathcal{CHA}$ and $a,b \in A$, we let $a \Rightarrow b := a \rightarrow (a \land b)$. 

\begin{theorem}\label{thm: connexiveheyting}
Let $\alga = \langle A,\land,\lor,\rightarrow,0,1 \rangle \in\mathcal{CHA}$. Then the algebra $\mathbb{H}(\alga)= \langle A,\land,\lor,\Rightarrow,0,1 \rangle$ is a Heyting algebra.
\end{theorem}

\begin{proof}
Let $\alga\in\mathcal{CHA}$. Since $\alga$ is a bounded distributive lattice, all we need to show is that for any $a,b,c \in A$, $a \land b \leq c$ iff $a \leq b \Rightarrow c$. The following chain of equivalences holds by Theorem \ref{marziano}, Lemma \ref{lem: filterscongr} and Theorem \ref{genoveffa}:
\begin{align*}
a\land b\leq c\quad\Longleftrightarrow& Fg(c) \subseteq Fg(a \land b)\\
\Longleftrightarrow&Fg(c) \subseteq Fg(a) \lor Fg(b)\\
\Longleftrightarrow&Fg(b) \ast Fg(c) \subseteq Fg(a)\\
\Longleftrightarrow&Fg(b \Rightarrow c) \subseteq Fg(a)\\
\Longleftrightarrow&a \leq b \Rightarrow c.
\end{align*}

\end{proof}

It is also true that any Heyting algebra has a connexive Heyting algebra term reduct. Hereafter, we denote by $\mathcal{HA}$ the variety of Heyting algebras. We follow the convention that the Heyting arrow $\Rightarrow$ binds less strongly than the lattice operations. Whenever $\mathbf{H} \in\mathcal{HA}$ and $a,b \in H$, we let $a \rightarrow b := (a \Rightarrow b) \land (\lnot a \Rightarrow \lnot b)$. 
\begin{theorem}\label{thm: heyting->connexive} Let $\mathbf{H}= \langle H,\land,\lor,\Rightarrow,0,1 \rangle \in \mathcal{HA}$. Then the algebra $\mathbb{C}(\mathbf{H})=\langle H,\land,\lor,\rightarrow,0,1 \rangle$ is a connexive Heyting algebra.
\end{theorem}
\begin{proof}
We show that $\mathbb{C}(\mathbf{H})$ satisfies C1 through C5 in Definition \ref{putacaso}. In so doing, we use without a mention some well-known properties of Heyting algebras. Throughout this proof, let $a,b,c$ be arbitrary elements of $H$. As regards C3:
\begin{align*}
(a\rightarrow b)\land a=& (a\Rightarrow b)\land(\neg a\Rightarrow \neg b)\land a\\
=& a\land b\land(\neg a\Rightarrow\neg b)\\
=& a\land b.
\end{align*}
We now move on to C4. We must establish that%
\[
\left(  a\Rightarrow b\right)  \wedge\left(  \lnot a\Rightarrow\lnot b\right)
\leq\left(  \left(  c\wedge a\right)  \Rightarrow\left(  c\wedge b\right)
\right)  \wedge\left(  \lnot\left(  c\wedge a\right)  \Rightarrow\lnot\left(
c\wedge b\right)  \right)  \text{.}%
\]

Observe first that%
\begin{align*}
\left(  a\Rightarrow b\right)  \wedge\left(  \lnot a\Rightarrow\lnot b\right)
\wedge c\wedge a  & =a\wedge b\wedge c\wedge\left(  \lnot a\Rightarrow\lnot
b\right)  \\
& =a\wedge b\wedge c\leq c\wedge b\text{.}%
\end{align*}

Hence%
\[
\left(  a\Rightarrow b\right)  \wedge\left(  \lnot a\Rightarrow\lnot b\right)
\leq  c\wedge a  \Rightarrow  c\wedge b  \text{.}%
\]

On the other hand, $0=c\wedge a\wedge\lnot\left(  c\wedge a\right)  $, whence
$c\wedge\lnot\left(  c\wedge a\right)  \leq\lnot a$. Thus%
\begin{align*}
c\wedge\lnot\left(  c\wedge a\right)  \wedge b\wedge\left(  a\Rightarrow
b\right)  \wedge\left(  \lnot a\Rightarrow\lnot b\right)    & =c\wedge
\lnot\left(  c\wedge a\right)  \wedge b\wedge\left(  \lnot a\Rightarrow\lnot
b\right)  \\
& \leq\lnot a\wedge b\wedge\left(  \lnot a\Rightarrow\lnot b\right)  \\
& =\lnot a\wedge b\wedge\lnot b=0\text{.}%
\end{align*}

Hence $\lnot\left(  c\wedge a\right)  \wedge\left(  a\Rightarrow b\right)
\wedge\left(  \lnot a\Rightarrow\lnot b\right)  \leq\lnot\left(  c\wedge
b\right)  $, and $\left(  a\Rightarrow b\right)  \wedge\left(  \lnot
a\Rightarrow\lnot b\right)  \leq\lnot\left(  c\wedge a\right)  \Rightarrow
\lnot\left(  c\wedge b\right)  $. Summing up, our claim follows. 

C5 is established similarly.

By the remarks following Theorem \ref{thm:ndcd-rdr}, to prove C1 it is enough to show that  $a\rightarrow b\leq(b\rightarrow c)\rightarrow (a\rightarrow c)$, which in turn holds if and only if:
\begin{enumerate} [(a)]
 \item $(a\Rightarrow b)\land(\neg a\Rightarrow\neg b)\leq (b\rightarrow c)\Rightarrow (a\rightarrow c)$, and
 \item $(a\Rightarrow b)\land(\neg a\Rightarrow\neg b)\leq \neg(b\rightarrow c)	\Rightarrow\neg(a\rightarrow c)$.
\end{enumerate}
Concerning (a), one has that $(a\Rightarrow b)\land(\neg a\Rightarrow\neg b)\land (b\Rightarrow c)\land (\neg b\Rightarrow\neg c)\leq (a\Rightarrow c)\land(\neg a\Rightarrow\neg c)$. So $(a\Rightarrow b)\land(\neg a\Rightarrow\neg b)\leq [(b\Rightarrow c)\land (\neg b\Rightarrow\neg c)]\Rightarrow [(a\Rightarrow c)\land(\neg a\Rightarrow\neg c)]=(b\rightarrow c)\Rightarrow(a\rightarrow c)$.\\
As regards (b), we have $(a\Rightarrow b)\land(\neg a\Rightarrow\neg b)\leq \neg(b\rightarrow c)\Rightarrow\neg(a\rightarrow c)\text{ iff } (a\Rightarrow b)\land(\neg a\Rightarrow\neg b)\land \neg((b\Rightarrow c)\land(\neg b\Rightarrow\neg c))\land (a\Rightarrow c)\land (\neg a\Rightarrow\neg c)=0$. We compute
\begin{align*}
(a\Rightarrow b)\land(\neg a\Rightarrow\neg b)\land \neg((b\Rightarrow c)\land(\neg b\Rightarrow\neg c))\land (a\Rightarrow c)\land (\neg a\Rightarrow\neg c)&\leq\\
(\neg b\Rightarrow \neg a)\land(\neg a\Rightarrow\neg b)\land((\neg b\Rightarrow\neg c)\Rightarrow\neg(b\Rightarrow c))\land (a\Rightarrow c)\land (\neg a\Rightarrow\neg c) &\leq\\
(\neg b\Rightarrow \neg c)\land(\neg a\Rightarrow\neg b)\land((\neg b\Rightarrow\neg c)\Rightarrow\neg(b\Rightarrow c))\land (a\Rightarrow c)&=\\
(\neg b\Rightarrow \neg c)\land(\neg a\Rightarrow\neg b)\land\neg(b\Rightarrow c)\land (a\Rightarrow c)&\leq\\
(\neg b\Rightarrow \neg c)\land(\neg a\Rightarrow\neg b)\land\neg(b\Rightarrow c)\land (\neg c\Rightarrow \neg a)&\leq\\
(\neg b\Rightarrow \neg c)\land (\neg c\Rightarrow\neg b)\land\neg(b\Rightarrow c)&=\\
(\neg b\Rightarrow \neg c)\land \neg\neg(b\Rightarrow c)\land\neg(b\Rightarrow c) &= 0.
\end{align*}
In fact, one has that, for any $a,b\in H$, $a\land\neg b\leq \neg(a\Rightarrow b)$ entails $\neg\neg(a\Rightarrow b)\leq\neg (a\land\neg b)=\neg b\Rightarrow\neg a$. Moreover, $\neg a\lor b\leq a\Rightarrow b$ implies $\neg b\Rightarrow\neg a=\llnot(\neg a\lor b)\leq\llnot(a\Rightarrow b)$.\\ 

Finally, concerning C2, we have that $(a\rightarrow b)\rightarrow\neg(a\rightarrow\neg b)=1$ if and only if
\begin{enumerate}[(a)]
 \item $(a\rightarrow b)\Rightarrow\neg (a\rightarrow\neg b) = 1$, and
 \item $\neg (a\rightarrow b)\Rightarrow \llnot (a\rightarrow\neg b) = 1$.
\end{enumerate}
As regards (a), we have that:
\begin{align*}
(a\rightarrow b)\Rightarrow\neg (a\rightarrow\neg b) = 1 &\text{ iff }(a\rightarrow b)\leq\neg (a\rightarrow\neg b)\\
&\text{ iff }(a\rightarrow b)\land (a\rightarrow\neg b)=0&\\
&\text{ iff }(a\Rightarrow b)\land(\neg a\Rightarrow\neg b)\land(a\Rightarrow\neg b)\land (\neg a\Rightarrow\llnot b) = 0\\
&\text{ iff } (a\Rightarrow (b\land\neg b))\land (\neg a\Rightarrow(\neg b\land\llnot b))=\neg a\land\llnot a=0.
\end{align*}
Since the last identity trivially holds, (a) is proved.\\
Concerning (b), it is easily seen that $\mathbf{H}$ satisfies, for any $a,b_{1},b_{2},c_{1},c_{2}\in H$:
\begin{equation}\label{eq: auxH}
((a\Rightarrow b_{1})\Rightarrow c_{1})\land((a\Rightarrow b_{2})\Rightarrow c_{2})\leq (a\Rightarrow (b_{1}\land b_{2}))\Rightarrow (c_{1}\land c_{2}).
\end{equation}
Now, we have $\neg(a\rightarrow b)\leq\llnot(a\rightarrow\neg b)$ iff $\neg(a\rightarrow b)\land\neg (a\rightarrow\neg b)=\neg((a\Rightarrow b)\land(\neg a\Rightarrow\neg b))\land\neg((a\Rightarrow \neg b)\land(\neg a\Rightarrow\llnot b))=0$. By \eqref{eq: auxH}, we compute
\begin{align*}
\neg((a\Rightarrow b)\land(\neg a\Rightarrow\neg b))\land\neg((a\Rightarrow \neg b)\land(\neg a\Rightarrow\llnot b))&=\\
((a\Rightarrow b)\Rightarrow\neg(\neg a\Rightarrow\neg b))\land((a\Rightarrow \neg b)\Rightarrow\neg(\neg a\Rightarrow\llnot b))&\leq\\
(a\Rightarrow (b\land\neg b))\Rightarrow(\neg(\neg a\Rightarrow\neg b)\land\neg(\neg a\Rightarrow\llnot b))&=\\
(\neg a\Rightarrow \neg(\neg a\Rightarrow\neg b))\land(\neg a\Rightarrow\neg(\neg a\Rightarrow\llnot b))&=\\
\neg(\neg a\land\neg b)\land\neg(\neg a\land\llnot b)&=\\
(\neg a\Rightarrow\llnot b)\land(\neg a\Rightarrow\lnot b)&=\\
\neg a\Rightarrow (\llnot b\land\lnot b)&=\llnot a.
\end{align*}
Similarly, one has: 
\begin{align*}
\neg((a\Rightarrow b)\land(\neg a\Rightarrow\neg b))\land\neg((a\Rightarrow \neg b)\land(\neg a\Rightarrow\llnot b))&=\\
((\neg a\Rightarrow\neg b)\Rightarrow\neg (a\Rightarrow b))\land ((\neg a\Rightarrow\llnot b)\Rightarrow\neg (a\Rightarrow \neg b))&\leq\\
(\neg a\Rightarrow(\neg b\land\llnot b))\Rightarrow(\neg (a\Rightarrow b)\land\neg (a\Rightarrow \neg b))&=\\
(\llnot a\Rightarrow\neg (a\Rightarrow b))\land(\llnot a\Rightarrow\neg (a\Rightarrow \neg b))&=\\
((a\Rightarrow b)\Rightarrow\neg a)\land((a\Rightarrow\neg b)\Rightarrow\neg a)&=\\
\neg(a\land b)\land\neg(a\land\neg b)&=\\
(a\Rightarrow\lnot b)\land(a\Rightarrow\llnot b)&=\neg a.
\end{align*}
Therefore, since we have $\neg((a\Rightarrow b)\land(\neg a\Rightarrow\neg b))\land\neg((a\Rightarrow \neg b)\land(\neg a\Rightarrow\llnot b))\leq\llnot a\land\lnot a=0$, the desired result obtains.
\end{proof}

\begin{theorem}\label{cor: termequiv}
The varieties $\mathcal{CHA}$ and $\mathcal{HA}$ are term equivalent. The term equivalence is implemented by the mutually inverse maps $\mathbb{H}$ of Theorem \ref{thm: connexiveheyting} and $\mathbb{C}$ of Theorem \ref{thm: heyting->connexive}. 
\end{theorem}

\begin{proof}
By Theorems \ref{thm: connexiveheyting} and \ref{thm: heyting->connexive}, the maps $\mathbb{H}$ and $\mathbb{C}$ are well-defined. It remains to be shown that they are mutually inverse, namely, that (a) if $\alga \in \mathcal{CHA}$, then $\mathbb{C}(\mathbb{H}(\alga))=\alga$, and (b) if $\algb \in \mathcal{HA}$, then $\mathbb{H}(\mathbb{C}(\algb))=\algb$.

First, we observe that any connexive Heyting algebra satisfies the following identity: 
\[x\land(z\rightarrow(x\land y))\approx x\land(z\rightarrow y).\]
Indeed, let $\alga \in \mathcal{CHA}$, and let $a,b,c \in A$.
By C6, $a\land (c\to (a\land b))=a\land ((a\land c)\to (a\land(a\land b)))=a\land ((a\land c)\to (a\land b))=a\land (c\to b)$. Next, we show:
\begin{itemize}
\item (P1) $a\rightarrow b=\mathrm{max}\{c\in A:a\land c\leq b\text{ and }\neg a\land c\leq\neg b\}$;
\item (P2) $a\rightarrow b=(a\rightarrow(a\land b))\land(\neg a\rightarrow(\neg a\land\neg b))$.
\end{itemize}
Note that $a\land(a\rightarrow b)\leq b$, by $\C 3$, and $\neg a\land(a\rightarrow b)\leq \neg a\land ((0\rightarrow a)\rightarrow(0\rightarrow b))=\neg a\land (\neg a\rightarrow\neg b)\leq \neg b$, by C3, Lemma \ref{lem:connexalgarithm1}.(2) and Lemma \ref{lem:arit}.(2)-(13). Now, let $c$ be such that $a\land c\leq b$ and $\neg a\land c\leq \lnot b$. By Lemma \ref{lem:connexalgarithm1}.(4) $c\land b\leq\llnot a$. Moreover, applying C4 and Lemma \ref{lem:arit}.(12)-(21), as well as the previous observations, $1 = a\rightarrow\llnot a\leq (a\land b\land c)\rightarrow (\llnot a\land b\land c)=(a\land c)\rightarrow(b\land c)$. Hence, in virtue of Lemma \ref{lem:connexalgarithm1}, $c\leq a\rightarrow(a\land c)\leq((a\land c)\rightarrow(b\land c))\rightarrow(a\rightarrow(b\land c))=1\rightarrow(a\rightarrow(b\land c))=a\rightarrow(b\land c)$. By the previously established identity, one has $c\leq a\rightarrow b$. Hence P1 follows.

Concerning P2, in light of P1 it suffices to show:
\begin{itemize}
    \item $(a\rightarrow(a\land b))\land(\neg a\rightarrow(\neg a\land\neg b)) \in \{c\in A:a\land c\leq b\text{ and }\neg a\land c\leq\neg b\}$;
    \item if $a\land c\leq b\text{ and }\neg a\land c\leq\neg b$, then $c \leq (a\rightarrow(a\land b))\land(\neg a\rightarrow(\neg a\land\neg b))$.
\end{itemize}
For the first bullet, we have that $a \land (a\rightarrow(a\land b))\land(\neg a\rightarrow(\neg a\land\neg b)) = a\land b \land (\neg a\rightarrow(\neg a\land\neg b)) \leq b$, and similarly $\lnot a \land (a\rightarrow(a\land b))\land(\neg a\rightarrow(\neg a\land\neg b)) \leq \lnot b$. For the second, if $a\land c\leq b$, then $a\land c\leq a \land b$, whence $c \leq a \to (a\land c) \leq a \to (a \land b)$. Similarly $c \leq \lnot a \to (\lnot a \land \lnot b)$, whence our conclusion follows.

Now, $\mathbb{C}(\mathbb{H}(\alga))=\alga$ is immediate by P2. In order to prove $\mathbb{H}(\mathbb{C}(\algb))=\algb$, just note that in $\algb$, for any $a,b \in B$, we have that $a \Rightarrow b = (a \Rightarrow (a \land b)) \land (\lnot a \Rightarrow \lnot (a \land b))$. 
\end{proof}

\subsection{The Boolean subvariety}

A noteworthy consequence of the results in the previous subsection is that there are continuum many subvarieties of $\mathcal{CHA}$, arranged in a lattice whose single atom is a term equivalent incarnation of the variety of Boolean algebras. We now aim at describing precisely this atom. Preliminarly, we prove the following lemma:

\begin{lemma}\label{char:symimpl}
Let $\alga$ be a connexive Heyting algebra. Then, for any $a\in A$, \[\llnot a = a\text{ if and only if }a\rightarrow b\leq b\rightarrow a,\text{ for any }b\in A.\]
\end{lemma}
\begin{proof}
The right-to-left direction follows from Lemma \ref{lem:connexalgarithm1}.(7). Conversely, note that by Lemmas \ref{lem:connexalgarithm1}.(1) and \ref{lem:arit}.(19)-(20), $b\rightarrow a =b\rightarrow\llnot a=\llnot(b\rightarrow a)=\llnot(a\rightarrow b)\geq a\rightarrow b.$
\end{proof}

We are now ready to characterise, in several different ways, the variety of connexive Heyting algebras that is term equivalent to the variety $\mc{BA}$ of Boolean algebras. Note that, alongside with the predictable demand that every element be Glivenko-closed (item 6), other equivalent conditions that axiomatise it relative to $\mathcal{CHA}$ include the symmetry of connexive implication (items 3, 4, 5) and its coincidence with material equivalence (item 2).  

\begin{lemma}\label{lem: charboolean} Let $\mc{V}$ be a subvariety of $\mathcal{CHA}$. The following are equivalent:
\begin{enumerate}
\item $\mc V$ is term equivalent to the variety $\mc{BA}$ of Boolean algebras;
\item $\mc V\models x\rightarrow y\approx (\neg x\lor y)\land (\neg y\lor x)$;
\item $\mc V\models x\rightarrow y\approx y\rightarrow x$;
\item $\mc V\models(x\rightarrow y)\rightarrow(y\rightarrow x)\approx 1$;
\item The following quasi-identity holds in $\mc V$:
\[x\rightarrow y = 1\quad\curvearrowright\quad y\rightarrow x=1;\]
\item $\mc V\models \neg\neg x\approx x$.
\end{enumerate}
\end{lemma}
\begin{proof} We first show that items (1), (2), (3), and (6) are all pairwise equivalent. By Lemma \ref{char:symimpl}, (3) is equivalent to (6), which is clearly equivalent to (1). If (2) holds, then in particular for all $a \in \alga \in \mathcal{V}$, $1 = a \to a = \lnot a \lor a$, and (1) follows. Finally, if (1) holds, then for all $a \in \alga \in \mathcal{V}$, $a \Rightarrow b = \lnot a \lor b$ and $\lnot a \Rightarrow \lnot b = b \Rightarrow a = \lnot b \lor a$, whence $a \to b = (a \Rightarrow b) \land (\lnot a \Rightarrow \lnot b) = (\lnot a \lor b) \land (\lnot b \lor a)$. Hence our claim is established.

(3) implies (4) by C7, and (4) implies (5) by Lemma \ref{lem:connexalgarithm1}.(1). Finally, (5) implies (6) as $a\rightarrow\llnot a= 1$ together with (5) entails that $\llnot a\rightarrow a = 1$, i.e. $a=\llnot a$.
 \end{proof}

Let us call $\mathcal{CBA}$ the variety which is axiomatised relative to $\mathcal{CHA}$ by any of these equivalent conditions; its members will be called \emph{connexive Boolean algebras}. We vigorously flag the fact that in $\mathcal{CBA}$ the connexive arrow denotes material equivalence, not material implication (which is denoted by the Heyting arrow).

The next example considers another subvariety of interest of $\mathcal{CHA}$: connexive G\"odel algebras.

\begin{example}Let $\mc{CGA}$ be the subvariety of $\mathcal{CHA}$ generated by all chains, whose relative equational basis with respect to $\mathcal{CHA}$ is the single identity
\[(x\rightarrow (x\land y))\lor (y\rightarrow (x\land y)) \approx 1.\label{sl}\tag{G}\]
This variety has been studied in \cite{Abad, DCV}. $\mc{CGA}$ is term equivalent to G\"odel algebras; it is not hard to show, using Theorem \ref{cor: termequiv}, that the equational basis provided in \cite{Abad} for $\mc{CGA}$ is equivalent (relative to $\mathcal{CHA}$) to \ref{sl}. Also, observe that any chain  $\mathbf{L}= \langle L,\land,\lor,0,1 \rangle$ can be uniquely equipped with a binary operation $``\rightarrow"$ such that $\langle L,\land,\lor,\rightarrow,0,1 \rangle \in\mc{CGA}$ by setting
\begin{equation}
    a\rightarrow b=
    \begin{cases}
      1 & \text{if } a \leq b\text{ and }\neg a=\neg b\\
      a\land b & \text{otherwise,}
    \end{cases}
  \end{equation} 
where
\begin{equation}
    \neg a=
    \begin{cases}
      1, & \text{if } a = 0 \\
      0 & \text{otherwise.}
    \end{cases}
  \end{equation}
\end{example}

We conclude this subsection by parlaying the above theorems into some Glivenko-style translation results. Let us set $\overline{A}=\{\llnot a:a\in A\}$, and consider the following binary operations over $\overline{A}$:\[x\Cap y=x\land y\quad\text{  and  }\quad x\Cup y=\llnot(x\lor y).\]
\begin{theorem}
Let $\alga\in\mathcal{CHA}$. Then the structure $\overline{\alga}=\langle \overline{A},\Cap,\Cup,\rightarrow,0,1 \rangle \in \mathcal{CBA}$. Moreover, the mapping $\llnot:A\rightarrow\overline{A}$ is an onto $\{\land,\rightarrow,0,1\}$-morphism.
\end{theorem}
\begin{proof}
Proving that $\overline{A}$ is closed under $\Cap$ and $\rightarrow$, and $\Cup$ is the l.u.b. in $\overline{A}$ is straightforward and is left to the reader. The remaining part of the statement follows by Lemma \ref{lem: charboolean}, Theorem \ref{cor: termequiv} and standard results concerning Heyting algebras.
\end{proof}

It is well known that, for any Heyting algebra $\mathbf{H}=\langle H,\land,\lor,\Rightarrow,0,1 \rangle$, the set $\mathrm{CC}(\mathbf{H})$ of closed and complemented (i.e. central) elements of $\mathbf{H}$ forms a sub-Heyting algebra of $\mathbf{H}$ which is a Boolean algebra. This fact together with Theorem \ref{cor: termequiv} yields the following
\begin{corollary} Let $\alga$ be a connexive Heyting algebra and let $\mathrm{CC}(\alga)$ be the set of closed and complemented elements of $\alga$. Then $\langle \mathrm{CC}(\alga),\land,\lor,\rightarrow,0,1 \rangle$ is a sub-connexive Heyting algebra of $\alga$ which is term-equivalent to a Boolean algebra.
\end{corollary}

\section{Connexive Heyting Logic}\label{ciaccaelle}

As a next item on our agenda, we capitalise on the previous results to obtain a deductive equivalence between the assertional logics of $\mathcal{CHA}$ and $\mathcal{HA}$. In the process, we obtain a Hilbert-style axiomatisation of the $1$-assertional logic of $\mathcal{CHA}$ and we gain insights that allow us to parlay the standard sequent calculus for intuitionistic logic into a calculus for this logic.

\subsection{An axiomatic calculus}

While faced with the problem of axiomatising $\mathrm{L}_{\mathcal{CHA}}$, one could be tempted to give it short shrift. Indeed, Theorem \ref{cor: termequiv} guarantees that $\mathcal{CHA}$ is term equivalent to $\mathcal{HA}$, and of course we know how to axiomatise the $1$-assertional logic of $\mathcal{HA}$, i.e., intuitionistic logic $\mathrm{IL}$. Why not simply apply the appropriate translation to the axioms of $\mathrm{IL}$? This approach, however, would be wrong-headed, as pointed out by Hiz \cite{Hiz} and several other authors after him \cite{Shapiro, Humbiz}. Hence, we have to proceed in a more roundabout way.

For a start, we introduce a new logic in the language $\mathcal{L}_{CH}$, whose consequence relation is determined by a certain Hilbert-style calculus. Then we use Theorem \ref{foscolo} to show that it coincides with $\mathrm{L}_{\mathcal{CHA}}$.

\begin{definition}
Let $\mathrm{CHL}=\left\langle \mathbf{Fm}_{\mathcal{L}_{CH}},\vdash _{\mathrm{%
CHL}}\right\rangle $, where $\vdash _{\mathrm{CHL}}$ is the derivability
relation of the Hilbert system with the following postulates (letting $\varphi \Rightarrow \psi$ be a shorthand for $\varphi \rightarrow (\varphi \land \psi)$):

\begin{description}
\item[CHL1] Any set of axioms and rules for positive logic (with implication replaced by
the defined connective $\Rightarrow $);

\item[CHL2] $\lnot \left( 0\wedge \varphi \right) $;

\item[CHL3] $\lnot \varphi \Rightarrow \left( 0\rightarrow \varphi \right) $;

\item[CHL4] $(\varphi \rightarrow \psi) \Rightarrow (\varphi \Rightarrow \psi) $;

\item[CHL5] $\left( \varphi \rightarrow \psi \right) \rightarrow \left( \left(
\psi \rightarrow \chi \right) \rightarrow \left( \varphi \rightarrow \chi
\right) \right) $;

\item[CHL6] $\left( \varphi \rightarrow \psi \right) \rightarrow \lnot \left(
\varphi \rightarrow \lnot \psi \right) $;

\item[CHL7] $\varphi \Leftrightarrow \psi \vdash (\varphi \to \chi) \Rightarrow (\psi \to \chi), (\chi \to \varphi) \Rightarrow (\chi \to \psi)$;

\item[CHL8] $\varphi \wedge \psi \Rightarrow \varphi \wedge \left( \varphi
\rightarrow \psi \right) $;

\item[CHL9] $\left( \varphi \rightarrow \psi \right) \Rightarrow \left( \left(
\varphi \wedge \chi \right) \rightarrow \left( \psi \wedge \chi \right)
\right) $;

\item[CHL10] $\left( \varphi \rightarrow \psi \right) \Rightarrow \left( \left(
\varphi \vee \chi \right) \rightarrow \left( \psi \vee \chi \right) \right) $%
.
\end{description}

\end{definition}

\begin{theorem}\label{manzanarre}
$\mathrm{CHL} = \mathrm{L}_{\mathcal{CHA}}$.
\end{theorem}

\begin{proof}
It is easy to see that the axioms CLH1-CLH10 evaluate at $1$ in any connexive Heyting algebra, and that the rule $\varphi, \varphi \Rightarrow \psi \vdash \psi$ preserves this property. For the converse direction, we resort to Theorem \ref{foscolo}. First, observe that the set $\{ \varphi \Rightarrow \psi, \psi \Rightarrow \varphi \}$ witnesses $1$-regularity for $\mathcal{CHA}$ and is a set of equivalence formulas for $\mathrm{L}_{\mathcal{CHA}}$ and $\mathcal{CHA}$. Thus, all we have to show is that the formulas and rules A1-A6 in Theorem \ref{foscolo} are derivable in $\mathrm{CHL}$.

As regards A1, A2, A3, and A5, they can be proved by means of the postulates of positive logic, hence of CLH1. The same can be said for A4, except for the rule
\[
\varphi_1 \Rightarrow \varphi_2, \psi_1 \Rightarrow \psi_2, \varphi_2 \Rightarrow \varphi_1, \psi_2 \Rightarrow \psi_1 \vdash \{ (\varphi_1 \rightarrow \psi_1) \Rightarrow (\varphi_2 \rightarrow \psi_2), (\varphi_2 \rightarrow \psi_2) \Rightarrow (\varphi_1 \rightarrow \psi_1) \},
\]
which can be proved by repeatedly applying CHL7. As for A6, if $\eta \approx \lambda$ is any of the identities C1-C5 in Definition \ref{putacaso}, $\rho (\eta, \lambda)$ they can be easily proved with the aid of CHL5, CHL6, CHL7-8, CHL9, and CHL10 respectively, as well as principles of positive logic. This leaves us with all $\rho (\eta, \lambda)$, where $\eta \approx \lambda$ is an identity axiomatising bounded distributive lattices. Again, CHL1 suffices to establish all the required theorems, except for $0 \land \varphi \Rightarrow 0$ and its converse $0 \Rightarrow 0 \land \varphi$. The former result follows from CHL2 and CHL4, whereas the latter is a consequence of CHL2 and CHL3.  
\end{proof}

Recalling Definition \ref{pincovalenza}, now we have all we need to prove the following

\begin{theorem}
$\mathrm{CHL}$ is deductively equivalent to intuitionistic logic $\mathrm{IL}
$. The equivalence is implemented by the translations $\tau ,\rho $ that leave all the connectives unaltered except for:%
\begin{eqnarray*}
 \varphi \rightarrow^{\tau} \psi &=&\left( \varphi \Rightarrow
\psi \right) \wedge \left( \lnot \varphi \Rightarrow \lnot \psi \right) 
\text{;} \\
 \varphi \Rightarrow^{\rho} \psi &=&\varphi \rightarrow \left(
\varphi \wedge \psi \right) \text{.}
\end{eqnarray*}
\end{theorem}

\begin{proof}
According to Definition \ref{pincovalenza}, we must show that for all $\Gamma \cup \left\{ \varphi \right\} \subseteq Fm_{%
\mathcal{L}_{CH}}$,

\begin{enumerate}
\item $\Gamma \vdash _{\mathrm{CLH}}\varphi $ iff $\tau \left( \Gamma
\right) \vdash _{\mathrm{IL}}\tau \left( \varphi \right) $;

\item $\tau \left( \rho \left( \varphi \right) \right) \dashv \vdash _{%
\mathrm{IL}}\varphi $.
\end{enumerate}
As regards (1), we have that:
\[
\begin{array}
[c]{llll}%
\Gamma\vdash_{\mathrm{CHL}}\varphi & \text{iff} & \left\{  \gamma
\approx1:\gamma\in\Gamma\right\}  \vdash_{\mathcal{CHA}}\varphi\approx1 &
\text{Thm. \ref{manzanarre}}\\
& \text{iff} & \left\{  \tau\left(  \gamma\right)  \approx1:\gamma\in
\Gamma\right\}  \vdash_{\mathcal{HA}}\tau\left(  \varphi\right)  \approx1 &
\text{Thm. \ref{cor: termequiv}}\\
& \text{iff} & \tau\left(  \Gamma\right)  \vdash_{\mathrm{IL}}\tau\left(
\varphi\right).   &
\end{array}
\]
For (2), it suffices to show that $\varphi \Rightarrow \psi$ is intuitionistically interderivable with $(\varphi \Rightarrow \varphi \land \psi) \land (\lnot \varphi \Rightarrow \lnot (\varphi \land \psi))$. We give an algebraic argument to that effect. Suppose $\alga$ is a Heyting algebra and $a,b \in A$. If $a \Rightarrow b = 1$, then $a \leq b$ and thus
\[
(a \Rightarrow a \land b) \land (\lnot a \Rightarrow \lnot (a \land b)) = (a \Rightarrow a) \land (\lnot a \Rightarrow \lnot a) = 1.
\]
Conversely, if $(a \Rightarrow a \land b) \land (\lnot a \Rightarrow \lnot (a \land b)) = 1$, then a fortiori $a \Rightarrow a \land b = 1$, hence $a \leq a \land b \leq b$, whereby $a \Rightarrow b = 1$.
\end{proof}

\subsection{Gentzen-style proof theory}

The deductive equivalence between $\mathrm{CHL}$ and $\mathrm{IL}$ certainly invites to piggyback on the existing proof systems for intuitionistic logic in order to obtain analytic calculi for our new logic. Although this may be a natural option, it need not be a straightforward, let alone a purely algorithmic, exercise. It is well-known that the most relevant properties of Gentzen calculi, like cut elimination, are by no means to be considered as intrinsic properties of a logic but are heavily sensitive to the particular presentation one chooses to adopt. 

In what follows, we introduce a sequent calculus $\mathtt{CHC}$ which is Gentzen algebraisable with $\mathcal{CHA}$ as equivalent variety semantics. As it will be clear below, $\mathtt{CHC}$ is virtually identical to the standard intuitionistic calculus $\mathtt{LJ}$, except for a different rule for introducing implication on the right, and an additional rule for introducing implication on the left. It is essentially different from the calculus for semi-intuitionistic logic (the logic corresponding to semi-Heyting algebras) given in \cite{Cassequent}, whose operational rules must be appropriately supplemented so as to guarantee the extra deductive power needed to prove the connexive axioms.\\

Hereafter, we retain our practice of denoting formulas in $Fm_{\mathcal{L}_{CH}}$ by $\varphi,\psi,\dots$, but also by $\alpha,\beta,\dots$, especially (but not only) when they are used as \emph{metaformulas} in rule schemata. Finite or empty sets of $\mathcal{L}_{CH}$-formulas are denoted by $\Gamma,\Delta,\dots$. We set $\neg\varphi:=\varphi\rightarrow 0$ , for any formula $\varphi$. A sequent is an ordered pair $\langle \Gamma,\Pi \rangle$ of finite sets of formulas where $\Pi$, called \emph{stoup}, is either empty or a singleton. As usual, a sequent $\langle \Gamma,\Pi \rangle$ is noted $\Gamma\rto\Pi$, and for any formulas $\varphi,\psi$, $\varphi\lto\rto\psi$ is short for the set $\{\varphi\rto\psi,\psi\rto\varphi\}$. $Seq_{\mc{L}_{CH}}$ will refer to the set of all sequents. If $\Gamma$ is a finite set of formulas, $\Gamma^{\land}$ stands for the conjunction of all formulas in $\Gamma$, associated to the left, if $\Gamma\neq\emptyset$, and $1$ otherwise. Similarly, if $\Pi$ is a stoup, $\Pi^{\lor}$ is the formula $\vrp$, if $\Pi^{\lor}=\{\vrp\}$, and $0$ otherwise.

The notions of an inference rule and a proof (or derivation) are the customary ones. If there exists a proof of $s$ from $S$, where $S\cup{s}\subseteq Seq_{\mc{L}}$, we will express this fact by $S\vdash_{\chc}s$. Observe that $\vdash_{\chc}$ is an abstract consequence relation according to Definition \ref{cioccoblocco}.

\begin{tcolorbox}[colback=white]
\begin{center}
\textbf{Axioms}\\
\vspace{4mm}
\AxiomC{}\id
\UnaryInfC{$\alpha\rto\alpha$}
\DisplayProof\qquad
\AxiomC{}\0
\UnaryInfC{$0\rto{}$}
\DisplayProof\qquad
\AxiomC{}\one
\UnaryInfC{${}\rto 1$}
\DisplayProof
\end{center}
\vspace{2mm}
\begin{center}
\textbf{Structural rules}\\
\vspace{4mm}
\AxiomC{$\Gamma\rto\Pi$}\wel
\UnaryInfC{$\alpha,\Gamma\rto\Pi$}
\DisplayProof\qquad
\AxiomC{$\Gamma\rto$}\wer
\UnaryInfC{$\Gamma\rto\alpha$}
\DisplayProof\qquad
\AxiomC{$\Gamma\rto\alpha$}
\AxiomC{$\alpha,\Delta\rto\Pi$}\cut
\BinaryInfC{$\Gamma,\Delta\rto\Pi$}
\DisplayProof
\end{center}
\vspace{2mm}
\begin{center}
\textbf{Operational Rules}\\
\vspace{4mm}
\AxiomC{$\alpha,\Gamma\rto\Pi$}
\UnaryInfC{$\alpha\land\beta,\Gamma\rto\Pi$}
\DisplayProof
\AxiomC{$\beta,\Gamma,\rto\Pi$}\landl
\UnaryInfC{$\alpha\land\beta,\Gamma\rto\Pi$}
\DisplayProof\qquad
\AxiomC{$\Gamma\rto\alpha$}
\AxiomC{$\Gamma\rto\beta$}\landr
\BinaryInfC{$\Gamma\rto\alpha\land\beta$}
\DisplayProof\\
\vspace{4mm}
\AxiomC{$\Gamma\rto\alpha$}
\UnaryInfC{$\Gamma\rto\alpha\lor\beta$}
\DisplayProof
\AxiomC{$\Gamma\rto\beta$}\lorr
\UnaryInfC{$\Gamma\rto\alpha\lor\beta$}
\DisplayProof\qquad
\AxiomC{$\alpha,\Gamma\rto\Pi$}
\AxiomC{$\beta,\Gamma\rto\Pi$}\lorl
\BinaryInfC{$\alpha\lor\beta,\Gamma\rto\Pi$}
\DisplayProof\\
\vspace{4mm}
\AxiomC{$\Gamma\rto\alpha$}
\AxiomC{$\Delta,\beta\rto\Pi$}\implla
\BinaryInfC{$\Delta,\Gamma,\alpha\rightarrow\beta\rto\Pi$}
\DisplayProof\qquad
\AxiomC{$\neg\alpha,\Gamma\rto\beta$}
\AxiomC{$\Delta,\alpha,\beta\rto{}$}\impllb
\BinaryInfC{$\Gamma,\Delta,\alpha\rightarrow\beta\rto{}$}
\DisplayProof\\
\vspace{4mm}
\AxiomC{$\alpha,\Gamma\rto\beta$}
\AxiomC{$\Delta,\lnot\alpha,\beta\rto{}$}\implr
\BinaryInfC{$\Gamma,\Delta\rto\alpha\rightarrow\beta$}
\DisplayProof\qquad
\end{center}
\end{tcolorbox}
It is easily seen by means of a routine argument that the inference rules ($\land$-l) can be equivalently replaced by the single rule
\begin{center}
\AxiomC{$\alpha,\beta,\Gamma\rto\Pi$}
\UnaryInfC{$\alpha\land\beta,\Gamma\rto\Pi$}
\DisplayProof
\end{center}
Therefore, in what follows, by ($\land$-l) we will mean an application of either ($\land$-l), or the above rule. We observe that some of the rules for connexive implication are neither separate, nor explicit in the sense of \cite{Wanpts}: They exhibit connectives other than the connexive arrow (i.e., the constant $0$), and they exhibit the arrow in their premiss sequents as well as in their conclusion sequents.

\begin{lemma}\label{gommapiuma}The following inference rules are derivable in $\chc$:
\vspace{4mm}
\begin{center}
\em{\AxiomC{$\Gamma\rto\alpha$}\negl
\UnaryInfC{$\Gamma,\lnot\alpha\rto{}$}
\DisplayProof\qquad
\AxiomC{$\Gamma,\alpha\rto{}$}\negr
\UnaryInfC{$\Gamma\rto\lnot\alpha$}
\DisplayProof\quad
\vspace{4mm}
{\AxiomC{$\Gamma,\alpha\rto{}$}
\AxiomC{$\Delta\rto\beta$}\em{\impllc}
\BinaryInfC{$\alpha\rightarrow\beta,\Gamma,\Delta\rto{}$}
\DisplayProof}\qquad
\AxiomC{$\alpha,\Gamma\rto\beta$}
\AxiomC{$\Delta,\neg\alpha,\beta\rto{}$}\em{\implld}
\BinaryInfC{$\Gamma,\Delta,\alpha\rightarrow\neg\beta\rto{}$}
\DisplayProof
}
\end{center}
\end{lemma}
\begin{proof}We confine ourselves to prove ($\neg$-r) and ($\rightarrow$-l(d)) leaving the remaining inference rules to the reader. Concerning ($\neg$-r), we have
\begin{prooftree}
\AxiomC{$\Gamma,\alpha\rto{}$}\wer
\UnaryInfC{$\Gamma,\alpha\rto 0$}
\AxiomC{}\0
\UnaryInfC{$0\rto{}$}\wel
\UnaryInfC{$0,\lnot\alpha\rto{}$}\implr
\BinaryInfC{$\Gamma\rto\alpha\rightarrow 0$}\eqr
\UnaryInfC{$\Gamma\rto\neg\alpha$}
\end{prooftree}
Furthermore, one can easily check that ($\neg$-l) can be proven by means of straightforward applications of ($\rightarrow$-l(a)) and ($0$), while ($\rightarrow$-l(c)) can be derived by applying (w-l) and ($\rightarrow$-l(b)). Finally, concerning ($\rightarrow$-l(d)), let us consider the following derivation:
\begin{prooftree}
\AxiomC{$\Delta,\neg\alpha,\beta\rto{}$}\negr
\UnaryInfC{$\Delta,\neg\alpha\rto\neg\beta$}
\AxiomC{$\Gamma,\alpha\rto\beta$}\negl
\UnaryInfC{$\Gamma,\alpha,\lnot\beta\rto{}$}\impllb
\BinaryInfC{$\Gamma,\Delta,\alpha\rightarrow\neg\beta\rto$}
\end{prooftree}
\end{proof}

\begin{lemma}\label{lem:c3c4gentzen}The following hold, for any $\varphi,\psi,\chi,\xi\in Fm_{\mc L_{CH}}$:
\begin{enumerate}
\item $\vchc(\varphi\rightarrow\psi)\land\varphi\lto\rto\varphi\land\psi$;
\item $\vchc\varphi\rightarrow\psi\rto(\varphi\land\chi)\rightarrow(\psi\land\chi)$;
\item $\vchc\varphi\rightarrow\psi\rto(\varphi\lor\chi)\rightarrow(\psi\lor\chi)$.
\end{enumerate}
\end{lemma}
\begin{proof}
Concerning (1), we have
\begin{prooftree}
\AxiomC{$\varphi\rto\varphi$}\landl
\UnaryInfC{$(\varphi\rightarrow\psi)\land\varphi\rto\varphi$}
\AxiomC{}\id
\UnaryInfC{$\varphi\rto\varphi$}
\AxiomC{}\id
\UnaryInfC{$\psi\rto\psi$}\implla
\BinaryInfC{$(\varphi\rightarrow\psi),\varphi\rto\psi$}\landl
\UnaryInfC{$(\varphi\rightarrow\psi)\land\varphi\rto\psi$}\landr
\BinaryInfC{$(\varphi\rightarrow\psi)\land\varphi\rto\varphi\land\psi$}
\end{prooftree}
Moreover, we have also
\begin{prooftree}
\AxiomC{}\id
\UnaryInfC{$\varphi\rto\varphi$}\landl
\UnaryInfC{$\varphi\land\psi\rto\varphi$}
\AxiomC{}\id
\UnaryInfC{$\psi\rto\psi$}\wel
\UnaryInfC{$\varphi,\psi \rto\psi$}
\AxiomC{}\id
\UnaryInfC{$\varphi\rto\varphi$}\negl
\UnaryInfC{$\varphi,\lnot\varphi\rto{}$}\wel
\UnaryInfC{$\varphi,\psi,\lnot\varphi\rto$}\implr
\BinaryInfC{$\varphi,\psi\rto\varphi\rightarrow\psi$}\landl
\UnaryInfC{$\varphi\land\psi\rto\varphi\rightarrow\psi$}\landr
\BinaryInfC{$\varphi\land\psi\rto(\varphi\rightarrow\psi)\land\varphi$}
\end{prooftree}
As regards (2), we have:
\begin{prooftree}
\AxiomC{}\id
\UnaryInfC{$\varphi\rto\varphi$}
\AxiomC{}\id
\UnaryInfC{$\psi\rto\psi$}\implla
\BinaryInfC{$\varphi\rightarrow\psi,\varphi\rto\psi$}\landl
\UnaryInfC{$\varphi\rightarrow\psi,\varphi\land\chi\rto\psi$}
\AxiomC{}\id
\UnaryInfC{$\chi\rto\chi$}\wel
\UnaryInfC{$\vrp\rightarrow\psi,\chi\rto\chi$}\landl
\UnaryInfC{$\vrp\rightarrow\psi,\vrp\land\chi\rto\chi$}\landr
\BinaryInfC{$\vrp\rightarrow\psi,\vrp\land\chi\rto\psi\land\chi$}
\AxiomC{$\mc D_{1}$}
\UnaryInfC{$\vrp\rightarrow\psi,\psi\land\chi,\lnot(\vrp\land\chi)\rto{}$}\implr
\BinaryInfC{$\vrp\rightarrow\psi\rto(\vrp\land\chi)\rightarrow(\psi\land\chi)$}
\end{prooftree}
where $\mc D_1$ has the following form:
\begin{prooftree}
\AxiomC{}\id
\UnaryInfC{$\vrp\rto\vrp$}\wel
\UnaryInfC{$\vrp,\psi\land\chi\rto\vrp$}
\AxiomC{}\id
\UnaryInfC{$\chi\rto\chi$}\wel
\UnaryInfC{$\vrp,\chi\rto\chi$}\landl
\UnaryInfC{$\vrp,\psi\land\chi\rto\chi$}\landr
\BinaryInfC{$\vrp,\psi\land\chi\rto\vrp\land\chi$}\negl
\UnaryInfC{$\vrp,\psi\land\chi,\neg(\vrp\land\chi)\rto{}$}
\AxiomC{}\id
\UnaryInfC{$\psi\rto\psi$}\landl
\UnaryInfC{$\psi\land\chi\rto\psi$}\impllc
\BinaryInfC{$\vrp\rightarrow\psi,\psi\land\chi,\neg(\vrp\land\chi)\rto{}$}
\end{prooftree}
Concerning (3), first let us consider the following proof $\mc D$
\begin{prooftree}
\AxiomC{}\id
\UnaryInfC{$\vrp\rto\vrp$}
\AxiomC{}\id
\UnaryInfC{$\psi\rto\psi$}\implla
\BinaryInfC{$\vrp\rightarrow\psi,\vrp\rto\psi$}\lorr
\UnaryInfC{$\vrp\rightarrow\psi,\vrp\rto\psi\lor\chi$}
\AxiomC{}\id
\UnaryInfC{$\chi\rto\chi$}\lorr
\UnaryInfC{$\chi\rto\psi\lor\chi$}\wel
\UnaryInfC{$\vrp\rightarrow\psi,\chi\rto\psi\lor\chi$}\lorl
\BinaryInfC{$\vrp\rightarrow\psi,\vrp\lor\chi\rto\psi\lor\chi$}
\end{prooftree}
Furthermore, we have
\begin{prooftree}
\AxiomC{$\mc D$}
\UnaryInfC{$\vrp\rightarrow\psi,\vrp\lor\chi\rto\psi\lor\chi$}
\AxiomC{}\id
\UnaryInfC{$\vrp\rto\vrp$}\lorr
\UnaryInfC{$\vrp\rto\vrp\lor\chi$}\negl
\UnaryInfC{$\vrp,\neg(\vrp\lor\chi)\rto{}$}
\AxiomC{}\id
\UnaryInfC{$\psi\rto\psi$}\impllc
\BinaryInfC{$\vrp\rightarrow\psi,\psi,\neg(\vrp\lor\chi)\rto{}$}
\AxiomC{}\id
\UnaryInfC{$\chi\rto\chi$}\lorr
\UnaryInfC{$\chi\rto\vrp\lor\chi$}\negl
\UnaryInfC{$\chi,\neg(\vrp\lor\chi)\rto{}$}\wel
\UnaryInfC{$\vrp\rightarrow\psi,\chi,\neg(\vrp\lor\chi)\rto{}$}\lorl
\BinaryInfC{$\vrp\rightarrow\psi,\psi\lor\chi,\neg(\vrp\lor\chi)\rto{}$}\implr
\BinaryInfC{$\vrp\rightarrow\psi\rto(\vrp\lor\chi)\rightarrow(\psi\lor\chi)$}
\end{prooftree}
\end{proof}
\begin{lemma}\label{lem:c1astc2astgentzen} The following hold, for any $\vrp,\psi,\chi\in Fm_{\mc L_{CH}}$:
\begin{enumerate}
\item $\vchc\vrp\rightarrow\psi,\vrp\rightarrow\neg\psi\rto{}$
\item $\vchc{}\rto(\vrp\rightarrow\psi)\rightarrow\neg(\vrp\rightarrow\neg\psi)$;
\item $\vchc \vrp\rightarrow\psi,\psi\rightarrow\chi\rto\vrp\rightarrow\chi$
\item $\vchc \vrp\rightarrow\psi\rto(\psi\rightarrow\chi)\rightarrow(\vrp\rightarrow\chi)$.
\end{enumerate}
\end{lemma}
\begin{proof}
Concerning (1), one has:
\begin{prooftree}
 \AxiomC{}\id
\UnaryInfC{$\vrp\rto\vrp$}
\AxiomC{}\id
\UnaryInfC{$\psi\rto\psi$}\implla
\BinaryInfC{$\vrp\rightarrow\psi,\vrp\rto\psi$}
\AxiomC{}\id
\UnaryInfC{$\vrp\rto\vrp$}\negl
\UnaryInfC{$\vrp,\neg\vrp\rto{}$}
\AxiomC{}\id
\UnaryInfC{$\psi\rto\psi$}\impllc
\BinaryInfC{$\vrp\rightarrow\psi,\neg\vrp,\psi\rto{}$}\implld
\BinaryInfC{$\vrp\rightarrow\psi,\vrp\rightarrow\neg\psi\rto{}$}
\end{prooftree}
Now, in order to prove (2), let us consider the following derivation $\mc D^\ast$
\begin{prooftree}
\AxiomC{}\id
\UnaryInfC{$\psi\rto\psi$}\wel
\UnaryInfC{$\vrp,\psi\rto\psi$}
\AxiomC{}\id
\UnaryInfC{$\vrp\rto\vrp$}\negl
\UnaryInfC{$\vrp,\neg\vrp\rto{}$}\wel
\UnaryInfC{$\vrp,\psi,\neg\vrp\rto{}$}\implr
\BinaryInfC{$\vrp,\psi\rto\vrp\rightarrow\psi$}\negl
\UnaryInfC{$\neg(\vrp\rightarrow\psi),\vrp,\psi\rto{}$}\negr
\UnaryInfC{$\neg(\vrp\rightarrow\psi),\vrp,\rto\neg\psi$}
\AxiomC{}\id
\UnaryInfC{$\vrp\rto\vrp$}\negl
\UnaryInfC{$\vrp,\neg\vrp\rto{}$}\wel
\UnaryInfC{$\neg\psi,\vrp,\neg\vrp\rto{}$}\wer
\UnaryInfC{$\neg\psi,\vrp,\neg\vrp\rto\psi$}
\AxiomC{}\id
\UnaryInfC{$\psi\rto\psi$}\negl
\UnaryInfC{$\psi,\neg\psi\rto{}$}\wel
\UnaryInfC{$\neg \vrp, \psi,\neg\psi\rto{}$}\implr
\BinaryInfC{$\neg\vrp,\neg\psi\rto\vrp\rightarrow\psi$}\negl
\UnaryInfC{$\neg\vrp,\neg\psi,\neg(\vrp\rightarrow\psi)\rto{}$}\implr
\BinaryInfC{$\neg(\vrp\rightarrow\psi)\rto\vrp\rightarrow\neg\psi$}
\end{prooftree}
Finally, one has
\begin{prooftree}
\AxiomC{Item (1)}
\UnaryInfC{$\vrp\rightarrow\psi,\vrp\rightarrow\neg\psi\rto{}$}\negr
\UnaryInfC{$\vrp\rightarrow\psi\rto\neg(\vrp\rightarrow\neg\psi)$}
\AxiomC{$\mc D^*$}
\UnaryInfC{$\neg(\vrp\rightarrow\psi)\rto\vrp\rightarrow\neg\psi$}\negl
\UnaryInfC{$\neg(\vrp\rightarrow\psi),\neg(\vrp\rightarrow\neg\psi)\rto{}$}\implr
\BinaryInfC{${}\rto(\vrp\rightarrow\psi)\rightarrow\neg(\vrp\rightarrow\neg\psi)$}
\end{prooftree}
Concerning (3), first let us consider the following proof $\mc D_2$
\begin{prooftree}
\AxiomC{}\id
\UnaryInfC{$\vrp\rto\vrp$}
\AxiomC{}\id
\UnaryInfC{$\psi\rto\psi$}\implla
\BinaryInfC{$\vrp\rightarrow\psi,\vrp\rto\psi$}
\AxiomC{}\id
\UnaryInfC{$\chi\rto\chi$}\implla
\BinaryInfC{$\vrp\rightarrow\psi,\psi\rightarrow\chi,\vrp\rto\chi$}
\end{prooftree}
Moreover, we have
\begin{prooftree}
\AxiomC{$\mc D_2$}
\UnaryInfC{$\vrp\rightarrow\psi,\psi\rightarrow\chi,\vrp\rto\chi$}
\AxiomC{}\id
\UnaryInfC{$\vrp\rto\vrp$}\negl
\UnaryInfC{$\vrp,\neg\vrp\rto{}$}
\AxiomC{}\id
\UnaryInfC{$\psi\rto\psi$}\impllc
\BinaryInfC{$\vrp\rightarrow\psi,\neg\vrp,\psi\rto{}$}
\AxiomC{}\id
\UnaryInfC{$\chi\rto\chi$}\impllc
\BinaryInfC{$\vrp\rightarrow\psi,\psi\rightarrow\chi,\chi,\neg\vrp\rto{}$}\implr
\BinaryInfC{$\vrp\rightarrow\psi,\psi\rightarrow\chi\rto\vrp\rightarrow\chi$}
\end{prooftree}
As regards (4), let us consider the following proof $\mc D_3$:
\begin{prooftree}
\AxiomC{$\mc D^{*}$}
\UnaryInfC{$\neg(\psi\rightarrow\chi)\rto\psi\rightarrow\neg\chi$}
\AxiomC{Item (3)}
\UnaryInfC{$\vrp\rightarrow\psi,\psi\rightarrow\lnot \chi\rto\vrp\rightarrow\lnot \chi$}
\AxiomC{Item (1)}
\UnaryInfC{$\vrp\rightarrow\chi,\vrp\rightarrow\neg\chi\rto{}$}\cut
\BinaryInfC{$\vrp\rightarrow\psi,\psi\rightarrow\neg\chi,\vrp\rightarrow\chi\rto{}$}\cut
\BinaryInfC{$\vrp\rightarrow\psi,\vrp\rightarrow\chi,\neg(\psi\rightarrow\chi)\rto{}$}
\end{prooftree}
Finally, one has:
\begin{prooftree}
\AxiomC{Item (3)}
\UnaryInfC{$\vrp\rightarrow\psi,\psi\rightarrow\chi\rto\vrp\rightarrow\chi$}
\AxiomC{$\mc D_3$}
\UnaryInfC{$\vrp\rightarrow\psi,\vrp\rightarrow\chi,\neg(\psi\rightarrow\chi)\rto{}$}\implr
\BinaryInfC{$\vrp\rightarrow\psi\rto(\psi\rightarrow\chi)\rightarrow(\vrp\rightarrow\chi)$}
\end{prooftree}
\end{proof}

\subsection{Gentzen algebraisability}

We are now ready to show that $\mathtt{CHC}$ is Gentzen algebraisable (according to Definition \ref{algebru}) with $\mc{CHA}$ as equivalent variety semantics. In other words, we show that there exists maps $\tau: Seq_{\mc{L}_{CH}}\rightarrow\mc{P}(Fm_{\mc{L}_{CH}}^{2})$ and $\rho:Fm_{\mc{L}_{CH}}^{2}\rightarrow\mc{P}(Seq_{\mc L_{CH}})$ such that, for any $S\cup\{s\}\subseteq Seq_{\mc L_{CH}}$, and $\vrp,\psi\in Fm_{\mc L_{CH}}$, one has:
\begin{enumerate}
\item $S\vchc s$ iff $\tau(S)\vdash_{\mc{CHA}}\tau(s)$;
\item $\vrp\approx\psi\dashv\vdash_{\mc{CHA}}\tau(\rho(\vrp\approx\psi))$;
\item $\tau(\sigma(s))= \sigma(\tau(s))$ and $\rho(\sigma(\varphi),\sigma(\psi))=\sigma(\rho(\varphi,\psi))$ for all substitutions $\sigma$ on $\mathbf{Fm}_{\mc L_{CH}}$, extended pointwise to $Fm_{\mc{L}_{CH}}^{2}$ and $Seq_{\mc L_{CH}}$.
\end{enumerate}
Given $\Gamma\rto\Pi\in Seq_{\mc L_{CH}}$ and $\vrp\approx\psi\in Fm_{\mc{L}_{CH}}^{2}$, we set
\[\tau(\Gamma\rto\Pi):=\Gamma^{\land}\leq\Pi^{\lor}\quad\text{ and }\quad\rho(\vrp\approx\psi):=\{\vrp\rto\psi,\psi\rto\vrp\}.\]
Given $S\subseteq Seq_{\mc L_{CH}}$, we set $\tau(S):=\{\Gamma^{\land}\leq\Pi^{\lor}:\Gamma\rto\Pi\in S\}$. Clearly, $\tau$ and $\rho$ commute with substitutions, whence (3) is satisfied.
A routine proof yields the following
\begin{lemma}\label{lem:normalsequents}
Let $S\cup\{\Delta\rto\Theta\}\subseteq Seq_{\mc L_{CH}}$. Then
\[S\vchc \Delta\rto\Theta\quad\text{ iff }\quad\{\Gamma^{\land}\rto\Pi^{\lor}:\Gamma\rto\Pi\in S\}\vchc\Delta^{\land}\rto\Theta^{\lor}.\]
\end{lemma}
\begin{proof}
Left to the reader.
\end{proof}
\begin{lemma} For all $\vrp,\psi\in Fm_{\mc L_{CH}}$, $\vrp\approx\psi\dashv\vdash_{\mc{CHA}}\tau(\rho(\vrp\approx\psi)).$
\end{lemma}
\begin{proof}
Just note that 
\begin{align*}
\vrp\approx\psi\dashv\vdash_{\mc{CHA}}\tau(\rho(\vrp\approx\psi))&\text{ iff }\vrp\approx\psi\dashv\vdash_{\mc{CHA}} \tau\{\vrp\rto\psi,\psi\rto\vrp\}\\
&\text{ iff }\vrp\approx\psi\dashv\vdash_{\mc{CHA}}\{\vrp\leq\psi,\psi\leq\vrp\}. 
\end{align*}
Since the last condition trivially holds, our result obtains.
\end{proof}
\begin{lemma}[Soundness]For all $S\cup\{s\}\subseteq Seq_{\mc L}$:
\[S\vchc s\text{ implies }\tau(S)\vdash_{\mc{CHA}}\tau(s).\]
\end{lemma}
\begin{proof}
Suppose that $s$ is $\Gamma\rto\Pi$. We prove the statement by induction on the length of the $\chc$-proof of $s$ from $S$. The base case is clear since, if $s$ is an axiom, then $\Gamma^\land \leq\Pi^{\lor}$ is $x\leq x$ which holds in $\mc{CHA}$, while if $s$ is an assumption the result is obvious. The induction step can be managed by distinguishing cases depending on the last rule applied in the derivation. We confine ourselves to the cases ($\rightarrow$-l(b)) and ($\rightarrow$-r), leaving the remaining cases to the reader. Suppose that $\Gamma\rto\Pi := \Gamma_1,\Delta \rto \varphi \to \psi$ has been obtained by $\Gamma_{1},\vrp \rto\psi$ and $\Delta,\psi,\neg\vrp\rto{}$ by means of an application of ($\rightarrow$-r). By induction hypothesis one has that $\tau(S)\vdash_{\mc{CHA}}\Gamma_{1}^{\land}\land\vrp\leq\psi$ and $\tau(S)\vdash_{\mc{CHA}}\Delta^{\land}\land\psi\land\neg\vrp\leq 0$. By Theorem \ref{cor: termequiv}, one has that $\tau(S)\vdash_{\mc{CHA}}\Gamma_{1}^{\land}\land\Delta^{\land}\leq (\vrp\Rightarrow\psi)\land(\neg\vrp\Rightarrow\neg\psi)=\vrp\rightarrow\psi$. Concerning the case ($\rightarrow$-l(b)), by induction hypothesis one has that $\tau(S)\vdash_{\mc{CHA}}\Gamma_{1}^{\land}\land\neg\vrp\leq\psi$ and $\tau(S)\vdash_{\mc{CHA}}\Delta^{\land}\land\vrp\land\psi\leq0$. By Lemma \ref{lem:connexalgarithm1}.(4), $\tau(S)\vdash_{\mc{CHA}}\Gamma_{1}^{\land}\land\neg\psi\land\neg\vrp\leq 0$ and $\tau(S)\vdash_{\mc{CHA}}\Delta^{\land}\land\vrp\leq\neg\psi$. Reasoning as above one has $\tau(S)\vdash_{\mc{CHA}}\Gamma_{1}^{\land}\land\Delta^{\land}\leq\vrp\rightarrow\neg\psi=\neg(\vrp\rightarrow\psi)$ by Lemma \ref{lem:arit}.(20). Therefore we conclude $\tau(S)\vdash_{\mc{CHA}}\Gamma_{1}^{\land}\land\Delta^{\land}\land(\vrp\rightarrow\psi)\leq\Gamma_{1}^{\land}\land\Delta^{\land}\land\neg\neg(\vrp\rightarrow\psi)\leq 0$.
\end{proof}
\begin{lemma}[Completeness] For all $S\cup\{s\}\subseteq Seq_{\mc L}$,
\[\tau(S)\vdash_{\mc{CHA}}\tau(s)\text{ implies }S\vchc s.\]
\end{lemma}
\begin{proof}
The proof follows a routine Lindenbaum-Tarski argument. Suppose contrapositively that $S\not\nvdash_{\chc} s$. We need a connexive Heyting algebra $\alga$ and a homomorphism $h:\mathbf{Fm}_{\mc{L}_{CH}}\rightarrow\alga$ such that $\tau(S)\subseteq \mathrm{ker}\, h$ while $\tau(s)\notin \mathrm{ker}\, h$. Let us denote by $T$ the smallest set of sequents containing $s$ and closed under $\vchc$. Moreover, for any $\vrp,\psi\in Fm_{\mc{L}_{CH}}$, we set $\vrp\lincong\psi$ if $\vrp\rto\psi,\psi\rto\vrp\in T$. We show that the desired algebra and homomorphism are $\mathbf{Fm}_{\mc{L}_{CH}}/\lincong$ and the natural homomorphism $x\mapsto x/\lincong$. To this aim we prove:
\begin{enumerate}
\item $\lincong$ is a congruence over $\mathbf{Fm}_{\mc{L}_{CH}}$, and
\item $\mathbf{Fm}_{\mc{L}_{CH}}/\lincong\in\mathcal{CHA}$.
\end{enumerate}
Concerning (1), note that $\lincong$ is obviously symmetric, reflexive and transitive by (id) and (cut). Now, in order to prove that $\lincong$ is compatible with operations, we show that for any $\vrp_{1},\vrp_{2},\psi_{1},\psi_{2}\in Fm_{\mc{L}_{CH}}$, $\vrp_{i}\lincong\psi_{i}$ ($i=1,2$) entails $(\vrp_{1}\star\vrp_{2})\lincong(\psi_{1}\star\psi_{2})$, for any $\star\in\{\land,\lor,\rightarrow\}$. Since the cases $\land,\lor$ are straightforward, we confine ourselves to $\rightarrow$. Let us consider the following derivation:
\begin{prooftree}
\AxiomC{$\psi_{1}\rto\vrp_{1}$}
\AxiomC{$\vrp_2\rto\psi_{2}$}\implla
\BinaryInfC{$\vrp_{1}\rightarrow\vrp_{2},\psi_{1}\rto\psi_{2}$}
\AxiomC{$\vrp_1\rto\psi_1$}\negl
\UnaryInfC{$\neg\psi_{1},\vrp_1\rto{}$}\wel
\UnaryInfC{$\neg\psi_{1},\vrp_1,\vrp_2\rto{}$}
\AxiomC{$\psi_2 \rto\vrp_2$}\wel
\UnaryInfC{$\psi_2 ,\lnot\vrp_{1}\rto\vrp_2$}\impllc
\BinaryInfC{$\vrp_{1}\rightarrow\vrp_{2},\psi_{2},\neg\psi_{1}\rto{}$}\implr
\BinaryInfC{$\vrp_{1}\rightarrow\vrp_{2}\rto\psi_{1}\rightarrow\psi_{2}$}
\end{prooftree}
Therefore $\vrp_{1}\rightarrow\vrp_{2}\rto\psi_{1}\rightarrow\psi_{2}\in T$. Similarly, one proves also that $\psi_{1}\rightarrow\psi_{2}\rto\vrp_{1}\rightarrow\vrp_{2}\in T$. We conclude that $\lincong$ is a congruence on $\mathbf{Fm}_{\mc{L}_{CH}}$.\\
As for (2), a routine proof show that the relation $\leq^{\mathbf{Fm}_{\mc{L}_{CH}}/\lincong}\subseteq (Fm_{\mc{L}_{CH}}/\lincong)^{2}$ such that $\vrp/\lincong\leq^{\mathbf{Fm}_{\mc{L}_{CH}}/\lincong}\psi/\lincong$ iff $\vrp\lto\rto\vrp\land\psi\subseteq T$ iff $\vrp\rto\psi\in T$ is indeed a lattice ordering. Furthermore, $\mathbf{Fm}_{\mc{L}_{CH}}/\lincong$ satisfies C1-C5 by Lemma \ref{lem:c3c4gentzen}, Lemma \ref{lem:c1astc2astgentzen} and the remarks following Theorem \ref{thm:ndcd-rdr}. Therefore we conclude $\mathbf{Fm}_{\mc L}/\lincong\in\mc{CHA}$. Now, it can be seen that $\tau(S)\subseteq\mathrm{ker}\, h$. Indeed,  making use of Lemma \ref{lem:normalsequents}, we have:
\begin{align*}
\tau(S)\subseteq\mathrm{ker}\, h &\text{ iff for any }\Gamma\rto\Pi\in S,\Gamma^{\land}\leq\Pi^{\lor}\in\mathrm{ker}\, h\\
 &\text{ iff for any }\Gamma\rto\Pi\in S,(\Gamma^{\land} \approx\Gamma^{\land}\land\Pi^{\lor})\in\mathrm{ker}\, h\\
&\text{ iff for any }\Gamma\rto\Pi\in S, \Gamma^{\land}\rto(\Gamma^{\land}\land\Pi^{\lor})\in T\\
&\text{ iff for any }\Gamma\rto\Pi\in S, \Gamma^{\land}\rto\Pi^{\lor}\in T\\
&\text{ iff for any }\Gamma\rto\Pi\in S, \Gamma\rto\Pi\in T.
\end{align*}
Given the way $T$ was defined, the last condition trivially holds. Similarly, one can show that $\tau(s)\notin \mathrm{ker}\, h$, since otherwise $S\vchc s$. Therefore our statement is proved.
\end{proof}
\subsection{Cut elimination}

Whether $\chc$ admits cut elimination is not an issue we can brush off by remarking that the intuitionistic calculus $\mathtt{LJ}$ is a cut-free calculus, and leaving it at that. Again, readers are warned that the existence of an algorithm for the elimination of cuts is not preserved by any of the relationships we have established in this paper. As a consequence, we must provide the required algorithm \textquotedblleft manually", as it were. This is the next item on our agenda.

\begin{theorem}
The calculus $\chc$ admits cut elimination.
\end{theorem}

\begin{proof}
The proof of this theorem has (nearly) the same structure as Gentzen's original proof for the intuitionistic calculus $\mathtt{LJ}$. In particular, it can be shown that $\chc$ can be equivalently formulated with sequents $\Gamma\rto\Pi$ consisting in a \emph{multiset} $\Gamma$ of formulas and a stoup $\Pi$, with an explicit contraction rule, and that in such a calculus the cut rule is equivalent to the \emph{mix rule}:
\begin{center}
\AxiomC{$\Gamma\rto\alpha$}
\AxiomC{$\Delta\rto\Pi$}\RightLabel{($mix_\alpha)$}
\BinaryInfC{$\Gamma,\Delta^{\ast \alpha}\rto\Pi$}
\DisplayProof
\end{center}
where $\Delta^{\ast \alpha}$ is $\Delta$ minus any occurrence of the mixformula $\alpha$. We focus on proofs $\mathcal{D}$ with a single final application of $mix_\alpha$, and we proceed by induction on the lexicographically ordered pair $\langle w(\mathcal{D}), r(\mathcal{D}) \rangle$, where:
\begin{itemize}
    \item $w(\mathcal{D})$ (the \emph{weight} of $\mathcal{D}$)\footnote{This definition of weight is essentially the one given by Negri and von Plato \cite{NvP} in establishing cut elimination for a certain calculus for intuitionistic logic.} is defined by induction on the construction of the mixformula $\alpha$: $w(\alpha)=0$ if $\alpha$ is the constant $0$, $w(\alpha)=1$ if $\alpha$ is a variable or the constant $1$, $w(\alpha)=w(\beta)+w(\gamma)+1$ if $\alpha$ has the form $\beta \ast \gamma$, with $\ast \in \{ \land, \lor, \rightarrow \}$.
    \item $r(\mathcal{D})$ (the \emph{rank} of $\mathcal{D}$) is customarily defined.
\end{itemize}
There are only two cases where the proof differs from the analogous proof for $\mathtt{LJ}$. Let $S_1$ and $S_2$ be the premisses of the final application of $mix_\alpha$ in $\mathcal{D}$. We must only consider the following cases:
\begin{enumerate}
    \item The case where $r(\mathcal{D})=2$, and both $S_1$ and $S_2$ are conclusions of an application of a logical rule, in which case $\alpha$ is principal in both such applications.
    \item The case where $r(\mathcal{D})>2$, the antecedent of $S_1$ does not contain $\alpha$, and $S_2$ is the conclusion of an application of a logical rule whose principal formula is $\alpha$.
\end{enumerate}
Unsurprisingly, we only address the case $\alpha = \beta \rightarrow \gamma$. If $\gamma$ is $0$, in light of Lemma \ref{gommapiuma}, the case can be dealt with using the cut elimination strategies for $\mathtt{LJ}$ (formulated with primitive negation) and obtaining thereby a reduction in the weight of $\mathcal{D}$. Thus, we lose no generality in supposing that $\alpha$ is not $0$. In Case (1), we suppose first that $S_1$ was obtained by ($\rightarrow$-r) and $S_2$ was obtained by ($\rightarrow$-l(a)):
\begin{prooftree}
\AxiomC{$\mathcal{D}_1$}
\UnaryInfC{$\beta, \Gamma \rto \gamma$}
\AxiomC{$\mathcal{D}_2$}
\UnaryInfC{$\Delta, \lnot \beta , \gamma \rto $}\implr
\BinaryInfC{$\Gamma,\Delta \rto \beta \rightarrow \gamma$}
\AxiomC{$\mathcal{D}_3$}
\UnaryInfC{$\Sigma \rto \beta$}
\AxiomC{$\mathcal{D}_4$}
\UnaryInfC{$\gamma, \Lambda \rto \Pi$}\implla
\BinaryInfC{$\beta \rightarrow \gamma, \Sigma, \Lambda \rto \Pi$}\RightLabel{($mix_\alpha)$}
\BinaryInfC{$\Gamma, \Delta, \Sigma, \Lambda \rto \Pi$}
\end{prooftree}
Observe that our assumption to the effect that $r(\mathcal{D})=2$ implies that $\Lambda = \Lambda^{\ast \alpha}$ and that $\Sigma = \Sigma^{\ast \alpha}$. Consider the following proof $\mathcal{D}_5$:
\begin{prooftree}
\AxiomC{$\mathcal{D}_3$}
\UnaryInfC{$\Sigma \rto \beta$}
\AxiomC{$\mathcal{D}_1$}
\UnaryInfC{$\beta, \Gamma \rto \gamma$}\RightLabel{$(mix_\beta)$}
\BinaryInfC{$\Gamma^{\ast \beta},\Sigma \rto \gamma$}
\AxiomC{$\mathcal{D}_4$}
\UnaryInfC{$\gamma, \Lambda \rto \Pi$}\RightLabel{$(mix_\gamma)$}
\BinaryInfC{$\Gamma^{\ast \beta}, \Sigma, \Lambda^{\ast \gamma} \rto \Pi$}
\end{prooftree}
The subproof of $\mathcal{D}_5$ ending with $\Gamma^{\ast \beta},\Sigma \rto \gamma$ can be replaced by a proof $\mathcal{D}_5'$, containing no mixes and ending with the same sequent, by Induction Hypothesis. Again by Induction Hypothesis, the result of this replacement can be converted into a proof without mix of $\Gamma^{\ast \beta}, \Sigma, \Lambda^{\ast \gamma} \rto \Pi$, from which a proof without mix of $\Gamma, \Sigma, \Lambda \rto \Pi$ can be obtained by successive applications of weakening.

Suppose next that $S_1$ was obtained by ($\rightarrow$-r) and $S_2$ was obtained by ($\rightarrow$-lb):
\begin{prooftree}
\AxiomC{$\mathcal{D}_1$}
\UnaryInfC{$\beta, \Gamma \rto \gamma$}
\AxiomC{$\mathcal{D}_2$}
\UnaryInfC{$\Delta, \lnot \beta , \gamma \rto $}\implr
\BinaryInfC{$\Gamma,\Delta \rto \beta \rightarrow \gamma$}
\AxiomC{$\mathcal{D}_3$}
\UnaryInfC{$\lnot \beta, \Lambda \rto \gamma$}
\AxiomC{$\mathcal{D}_4$}
\UnaryInfC{$\beta, \gamma, \Sigma \rto $}\implla
\BinaryInfC{$\beta \rightarrow \gamma, \Sigma, \Lambda \rto $}\RightLabel{($mix_\alpha)$}
\BinaryInfC{$\Gamma, \Delta, \Sigma, \Lambda \rto $}
\end{prooftree}
Observe again that our assumption to the effect that $r(\mathcal{D})=2$ implies that $\Lambda = \Lambda^{\ast \alpha}$ and that $\Sigma = \Sigma^{\ast \alpha}$. Consider the following proof $\mathcal{D}_5$:
\begin{prooftree}
\AxiomC{$\mathcal{D}_2$}
\UnaryInfC{$\Delta, \lnot \beta, \gamma \rto$}
\AxiomC{$\mathcal{D}_3$}
\UnaryInfC{$\lnot \beta, \Lambda \rto \gamma $}\RightLabel{$(mix_\gamma)$}
\BinaryInfC{$\Lambda, \Delta^{\ast \gamma}, \lnot \beta \rto$}
\AxiomC{$\mathcal{D}_1$}
\UnaryInfC{$\beta, \Gamma \rto \gamma$}
\AxiomC{$\mathcal{D}_4$}
\UnaryInfC{$\beta, \gamma, \Sigma \rto $}\RightLabel{$(mix_\gamma)$}
\BinaryInfC{$\beta, \Gamma, \Sigma^{\ast \gamma} \rto $}\negr
\UnaryInfC{$\Gamma, \Sigma^{\ast \gamma} \rto \lnot \beta$}\RightLabel{$(mix_{\lnot \beta)}$}
\BinaryInfC{$\Gamma, \Sigma^{\ast \gamma}, \Lambda^{\ast \beta}, \Delta^{\ast \gamma \ast \lnot \beta} \rto $}
\end{prooftree}
Remark that the subproofs of $\mathcal{D}_5$ respectively ending with $\Lambda, \Delta^{\ast \gamma}, \lnot \beta \rto$, $\beta, \Gamma, \Sigma^{\ast \gamma} \rto $, and $\Gamma, \Sigma^{\ast \gamma}, \Lambda^{\ast \beta}, \Delta^{\ast \gamma \ast \lnot \beta} \rto $ have strictly smaller weights than $\mathcal{D}$, since we have assumed that $\gamma$ is not $0$. Hence, reasoning as above and using the Inductive Hypothesis several times, we conclude that there is a proof of $\Gamma, \Sigma^{\ast \gamma}, \Lambda^{\ast \beta}, \Delta^{\ast \gamma \ast \lnot \beta} \rto $ containing no mixes, and by successive applications of weakening we end up proving $\Gamma, \Sigma, \Lambda, \Delta \rto $.

As regards Case (2), the only interesting subcase is as follows:
\begin{prooftree}
\AxiomC{$\mathcal{D}_1$}
\UnaryInfC{$\Gamma \rto \beta \rightarrow \gamma$}
\AxiomC{$\mathcal{D}_2$}
\UnaryInfC{$\lnot \beta, \Delta \rto \gamma$}
\AxiomC{$\mathcal{D}_3$}
\UnaryInfC{$\Sigma, \beta, \gamma \rto $}\impllb
\BinaryInfC{$\beta \rightarrow \gamma, \Delta, \Sigma \rto $}\RightLabel{$(mix_\alpha)$}
\BinaryInfC{$\Gamma, \Delta^{\ast \alpha}, \Sigma^{\ast \alpha} \rto $}
\end{prooftree}
We first trade $\mathcal{D}$ for two proofs with a single final application of mix, call them $\mathcal{D}'$ and $\mathcal{D}''$ respectively, having the same weight as $\mathcal{D}$ and a strictly smaller rank:
\begin{prooftree}
\AxiomC{$\mathcal{D}_1$}
\UnaryInfC{$\Gamma \rto \beta \rightarrow \gamma$}
\AxiomC{$\mathcal{D}_2$}
\UnaryInfC{$\lnot \beta, \Delta \rto \gamma$}\RightLabel{$(mix_\alpha)$}
\BinaryInfC{$\lnot \beta, \Gamma, \Delta^{\ast \alpha} \rto \gamma $}
\end{prooftree}
\begin{prooftree}
\AxiomC{$\mathcal{D}_1$}
\UnaryInfC{$\Gamma \rto \beta \rightarrow \gamma$}
\AxiomC{$\mathcal{D}_3$}
\UnaryInfC{$\Sigma, \beta, \gamma \rto $}\RightLabel{$(mix_\alpha)$}
\BinaryInfC{$\Gamma, \Sigma^{\ast \alpha}, \beta, \gamma \rto $}
\end{prooftree}
By Inductive Hypothesis, there are $\mathcal{D}'''$ and $\mathcal{D}''''$ containing no mixes, respectively ending with $\lnot \beta, \Gamma, \Delta^{\ast \alpha} \rto \gamma $ and $\Gamma, \Sigma^{\ast \alpha}, \beta, \gamma \rto $. Hence the following proof $\mathcal{D}_4$:
\begin{prooftree}
\AxiomC{$\mathcal{D}_1$}
\UnaryInfC{$\Gamma \rto \beta \rightarrow \gamma$}
\AxiomC{$\mathcal{D}'''$}
\UnaryInfC{$\lnot \beta, \Gamma, \Delta^{\ast \alpha} \rto \gamma$}
\AxiomC{$\mathcal{D}''''$}
\UnaryInfC{$\Gamma, \Sigma^{\ast \alpha}, \beta, \gamma \rto $}\impllb
\BinaryInfC{$\beta \rightarrow \gamma, \Gamma, \Delta^{\ast \alpha}, \Sigma^{\ast \alpha} \rto $}\RightLabel{$(mix_\alpha)$}
\BinaryInfC{$\Gamma, \Delta^{\ast \alpha}, \Sigma^{\ast \alpha} \rto $}
\end{prooftree}
has a right rank equal to $1$ (for $\Gamma$ does not contain $\alpha$), and a rank strictly less than that of $\mathcal{D}$. Since the weights of $\mathcal{D}_4$ and $\mathcal{D}$ are the same, we have got every right to apply the Inductive Hypothesis and also this subcase is settled.
\end{proof}

\section{Philosophical upshots}\label{philup}

\subsection{BHK interpretation of connexive implication}

The connection we have discovered between intuitionistic logic and a certain connexive logic unearths a gravy train in terms of opportunities to shed new light on the very idea of connexivity. Together with classical logic, intuitionistic logic is perhaps the best understood logical system as regards its philosophical foundations. In particular, the celebrated \emph{BHK (Brouwer-Heyting-Kolmogorov) interpretation} (for which see e.g. \cite{TvD}) allows the intuitionistic logician to assign a constructive, computational meaning to the intuitionistic connectives and quantifiers. Via our deductive equivalence, we can parlay this semantics of proofs into a constructive interpretation of the connexive conditional. 

For a start, recall the BHK interpretation of conjunction, implication, negation and falsity:

\begin{itemize}
\item a proof of $\varphi \wedge \psi $ is a pair consisting in a proof of $%
\varphi $ and a proof of $\psi $;

\item a proof of $\varphi \Rightarrow \psi $ is a function that converts any
(hypothetical) proof of $\varphi $ into a proof of $\psi $;

\item there is no proof of $0$;

\item a proof of $\lnot \varphi :=\varphi \Rightarrow 0$ is a function that converts any
(hypothetical) proof of $\varphi $ into a proof of $0$; since, however, there is no proof of $0$, a proof of $\lnot \varphi$ amounts to a refutation of $\varphi $.

\end{itemize}

The given clause for negation has been criticised by Wansing \cite{Wan} because, in the BHK framework, an intuitionistically negated formula $\lnot \varphi$ is valid if and only if there exists a construction that outputs a nonexistent object, namely a proof of $0$, when applied to a proof of $\varphi$, a condition that can be satisfied only vacuously for unprovable formulas. Interestingly, this objection is echoed by Kapsner \cite{Kapsbook} in his defence of Aristotle's law from the alleged counterexamples arising in correspondence of unsatisfiable formulas. According to Kapsner, such putative counterexamples rest on \textquotedblleft empty promise conversions" very much like the intuitionistic falsifications deplored by Wansing (see also \cite{WOF}).

Here, on the other hand, we do not intend to take issue with the standard BHK interpretation of logical constants -- rather, we aim at reading off its clauses a possible computational meaning for the connexive implication of $\mathrm{CHL}$. A disclaimer is of course in order: We do not claim by any means that the suggestions that follow apply to \emph{any} connexive implication. It is unlikely, for example, that they can somehow relate to the implications studied within the different traditions stemming from Nelson \cite{Nelson}, Angell and McCall \cite{Angell}, De Finetti, Cooper and Cantwell \cite{ERS}, or the so-called \textquotedblleft Bochum plan" \cite{Wanbochum}, all of which are based on quite different intuitions. Other approaches, like the Boolean connexive logics of Jarmuzek and Malinowski \cite{JM}, and in particular the connexive logic of content equality by Estrada Gonzalez and Klonowski \cite{EGK}, may on the contrary stand better chances to ensconce themselves into the interpretation we suggest.

Thus, recall that $\varphi \rightarrow \psi $ can be defined in $\mathrm{IL}
$ as $\left( \varphi \Rightarrow \psi \right) \wedge \left( \lnot \varphi
\Rightarrow \lnot \psi \right) $, that $\lnot \varphi \Rightarrow \lnot \psi 
$ is intuitionistically equivalent to $\psi \Rightarrow \lnot \lnot \varphi $, and that $\varphi \Rightarrow 0$ is equivalent in both $\mathrm{CHL}$ and $\mathrm{IL}$ to $%
\varphi \rightarrow 0$. The standard BHK interpretation of the intuitionistic connectives appearing in the compound formula that interprets the connexive conditional translates into the following reading for $\varphi \to \psi$:

\begin{itemize}
\item a proof of $\varphi \rightarrow \psi $ is a pair consisting in a
function that converts any (hypothetical) proof of $\varphi $ into a proof
of $\psi $, and a function that converts any (hypothetical) proof of $\psi $
into a refutation of the refutation of $\varphi $.
\end{itemize}

A proof of a connexive implication $\varphi \rightarrow \psi $ can be seen as consisting of two different parts: A constructively acceptable proof of $\psi$ on the assumption that $\varphi$, and a weaker, \emph{classically} (but not perforce intuitionistically) valid proof of $\varphi$ on the assumption that $\psi$. It remains to be seen whether the weak asymmetry that distinguishes the different directions of such a \textquotedblleft quasi-equivalence" is sufficient to qualify our connective as a full-blooded conditional, as opposed to a biconditional in disguise. This misgiving certainly deserves a fuller discussion, which we defer to future research.

\subsection{On superconnexivity}

It is now time to take stock with respect to the idea of strong connexivity. At the outset, we sympathetically endorsed Kapsner's quest for logics that are not only legally connexive, in so far as they abide by the minimal requirements to be certified as such, but also have the concept that $\varphi \rightarrow \lnot \varphi $ is a sort of \textquotedblleft connexive contradiction" -- and that $\varphi \rightarrow \psi $ and $\varphi
\rightarrow \lnot \psi $ are a sort of \textquotedblleft connexive contraries" -- deeply ingrained in their semantics. $\mathrm{CHL}$ is strongly connexive in precisely this sense, since $\varphi \rightarrow \lnot \varphi $ is equivalent to $\lnot \varphi \land \lnot \lnot \varphi$, and hence unsatisfiable, while $\varphi \rightarrow \psi $ and $\varphi
\rightarrow \lnot \psi $ are respectively equivalent to $\varphi \Rightarrow \psi \land \lnot \varphi \Rightarrow \lnot \psi$ and to $\varphi \Rightarrow \lnot \psi \land \lnot \varphi \Rightarrow \lnot \lnot \psi$, and hence non simultaneously satisfiable.

Interestingly, Kapsner \cite{Kapstrong} also entertains, but ultimately rejects, a natural option for attaining strong connexivity by capturing \emph{in the object language} the unsatisfiability of $\varphi \rightarrow \lnot \varphi $, or the
non-simultaneous satisfiability of $\varphi \rightarrow \psi $ and $\varphi \rightarrow \lnot \psi $. He states some explosion-like \emph{superconnexive} principles, including:

\begin{itemize}
\item $\left( \varphi \rightarrow \lnot \varphi \right) \rightarrow \psi $
(Super-Aristotle 1)

\item $\left( \varphi \rightarrow \psi \right) \rightarrow \left( \left(
\varphi \rightarrow \lnot \psi \right) \rightarrow \chi \right) $
(Super-Boethius 1)
\end{itemize}

Yet, these principles are dumped because they lead to triviality given a modicum of assumptions. Very recently, however, Kapsner and Omori \cite{Kapsomo} have attempted to revisit the superconnexive insight. Their goal, in a nutshell, is to salvage the spirit of superconnexivity by slightly weakening the letter of it. The concept behind the standard principle of explosion can be pinned down in different ways -- by the demand that a contradiction entail \emph{any} sentence, or perhaps by the demand that a contradiction entail a designated absurdity, like the falsum constant. This may make no difference in most contexts, but sometimes it does (like in some relevant logics). Analogously, one might envisage the thought that a connexive contradiction need not entail any sentence whatsoever, but only the falsum. This naturally leads to the following \emph{super-Bot-connexive} principles:

\begin{itemize}
\item $\left( \varphi \rightarrow \lnot \varphi \right) \rightarrow 0$
(Super-Bot-Aristotle 1)

\item $\left( \varphi \rightarrow \psi \right) \rightarrow \left( \left(
\varphi \rightarrow \lnot \psi \right) \rightarrow 0\right) $
(Super-Bot-Boethius 1)
\end{itemize}

Unlike the original superconnexive principles, these weaker laws are not so easily trivialised: Indeed, Kapsner and Omori point out that they are consistent with a number of axiomatic frameworks.

Maybe, though, super-Bot-connexivity is an unnecessary retreat. Perhaps superconnexivity was abandoned too swiftly, while it was only in need of some rephrasing. Let us consider Super-Aristotle 1, by way of example. Connexive implication occurs twice therein -- once in the formulation of the connexive contradiction $\varphi \rightarrow \lnot \varphi$, and once to signal that such a contradiction explosively implies any old formula. In $\mathrm{CHL}$, we have an intuitionistic conditional that coexists with the connexive one. Can we avoid trivialisation by replacing one of the occurrences of implication in Super-Aristotle 1 by its intuitionistic counterpart? If we want to do so in a principled way, and not merely as a means to the end of consistency preservation, we ought to look at the BHK interpretation of these connectives, as spelt out in the previous subsection. The former occurrence of the arrow in Super-Aristotle 1 can't be anything but a connexive implication -- else, how could $\varphi \rightarrow \lnot \varphi$ express the idea of a \emph{connexive} contradiction? The latter occurrence, on the other hand, is much more plausibly construed as an intuitionistic conditional. If it can't be the case that $\varphi
\rightarrow \lnot \varphi $, any hypothetical proof of this fact should
(vacuously) yield a proof of an arbitrary $\psi $. But there's no reason to
expect that any hypothetical proof of some $\psi $ would yield a refutation
of a refutation of $\varphi \rightarrow \lnot \varphi $...

For these reasons, we argue that the ideal object-language analogues of the
strong connexive unsatisfiability principles are obtained by tweaking as follows the original superconnexive principles:

\begin{itemize}
\item $\left( \varphi \rightarrow \lnot \varphi \right) \Rightarrow \psi $
(mixed Super-Aristotle 1)

\item $\left( \varphi \rightarrow \psi \right) \Rightarrow \left( \left(
\varphi \rightarrow \lnot \psi \right) \Rightarrow \chi \right) $ (mixed
Super-Boethius 1)
\end{itemize}

Of course, this is only a preliminary suggestion in need of a deeper scrutiny, which we intend to bring forth in the future.

\section{Conclusions and open problems}\label{openprob}

The connection we found between a certain connexive logic and a time-honoured, well-understood logic like $\mathrm{IL}$ opens promising avenues of research. We list hereafter some problems one could naturally address. 

\begin{itemize}
\item \emph{Develop more proof systems for $\mathrm{CHL}$.} We used the term equivalence between $\mathrm{CHL}$ and $\mathrm{IL}$ to rejig the sequent calculus for the latter into a corresponding calculus for connexive implication. Something analogous can certainly be done for the other calculi (e.g. natural deduction systems) available for intuitionistic logic.

\item \emph{Study the extensions of $\mathrm{CHL}$.} The study of intermediate logics (logics that lie between $\mathrm{IL}$ and classical logic in terms of deductive strength) is a fruitful and amply trodden area of investigation. Via our translational equivalence, we get uncountably many logics between $\mathrm{CHL}$ and classical logic. It would be interesting to explore their properties and to assess their significance.

\item \emph{Clarify the relationships between $\mathrm{CHL}$ and other connexive
logics.} In particular, one should focus on other connexive logic based on positive logic, the prime example being Wansing's $\mathrm{C}$ \cite{Wanbochum}. It would also be desirable to shed some further light on the relationships between superconnexivity, super-Bot connexivity and mixed superconnexivity.

\item \emph{Make sense of other features inherited from $\mathrm{IL}$.} Intuitionistic logic is extremely pliant to different semantical analysises, in terms of Kripke models, topological semantics, etc. Perhaps a treatment of our connexive implication within these frameworks could better enlighten its meaning and conceptual significance.
\end{itemize}

\begin{acknowledgement}
A preliminary version of this material has been presented at Trends in Logic 21 -- Frontiers of Connexive Logic, Bochum, December 6-8, 2021. Thanks are due to the organisers of that conference and to all participants for their insightful comments. In particular, we are grateful to Luis Estrada Gonzalez, Andi Kapsner, Jacek Malinowski, Hitoshi Omori, and Heinrich Wansing, to whom we are indebted for several stimulating discussions. We gratefully acknowledge the support of Fondazione di Sardegna within the project \textquotedblleft Resource sensitive reasoning and logic\textquotedblright, Cagliari, CUP: F72F20000410007 and of MIUR within the projects PRIN 2017: \textquotedblleft Theory and applications of resource sensitive logics\textquotedblright, CUP: 20173WKCM5 and \textquotedblleft Logic and cognition. Theory, experiments, and applications\textquotedblright, CUP: 2013YP4N3. 
\end{acknowledgement}


\begin{thebibliography}{99}


\bibitem {Abad}Abad M., Cornejo J.M., Diaz Varela J.P., \textquotedblleft The variety generated by semi-Heyting chains", \emph{Soft Computing}, 15, 2011, pp. 721-–728.

\bibitem {DCV}Abad M., Cornejo J.M., Diaz Varela J.P., \textquotedblleft Semi-Heyting algebras term-equivalent to Goedel algebras", \emph{Order}, 30, 2013, pp. 625-–642.

\bibitem {OSV3}Aglian\`{o} P., Ursini A., \textquotedblleft On subtractive varieties
III:\ From ideals to congruences", \emph{Algebra Universalis}, 37, 1997, pp. 296-333.

\bibitem{Angell}Angell R.B., \textquotedblleft A propositional logic with subjunctive conditionals", \emph{Journal of Symbolic Logic}, 27, 3, 1962, pp. 327–343.

\bibitem {BJ}Blok W.J., J\'{o}nsson B., \textquotedblleft Equivalence of consequence
operations", \emph{Studia Logica}, 83, 2006, pp. 91--110.

\bibitem{EDPC2} Blok W.J., K\"ohler P., Pigozzi D., \textquotedblleft On the structure of varieties with equationally definable principal congruences II", \emph{Algebra Universalis}, 18, 1984, pp. 334-379.

\bibitem {EDPC1}Blok W. J., Pigozzi D., \textquotedblleft On the structure of
varieties with equationally definable principal congruences
I\textquotedblright, \emph{Algebra Universalis}, 15, 1982, pp. 195-227.

\bibitem {BlokP}Blok W.J., Pigozzi D., \emph{Algebraizable Logics}, Memoirs
of the AMS, number 396, American Mathematical Society, Providence, RI, 1989.

\bibitem {EDPC3} Blok W. J.,  Pigozzi D., \textquotedblleft On the structure of varieties with equationally definable principal congruences III", \emph{Algebra Universalis}, 32, 1994, pp. 545--608.

\bibitem {EDPC4}Blok W.J., Pigozzi D., \textquotedblleft On the structure of varieties with
equationally definable principal congruences IV", \emph{Algebra Universalis},
31, 1994, pp. 1-35.

\bibitem {BlokR}Blok W.J., Raftery J.G., \textquotedblleft Ideals in quasivarieties of
algebras", in X. Caicedo and C.H. Montenegro (Eds.), \emph{Models, Algebras
and Proofs}, Dekker, New York,  1999, pp. 167--186.

\bibitem {BR+}Blok W.J., Raftery J.G., \textquotedblleft Assertionally equivalent quasivarieties", \emph{International Journal of Algebra and Computation}, 18,  2008, pp. 589--681.

\bibitem {Burris}Burris S., Sankappanavar H.P., \emph{A Course in Universal Algebra}, Springer, Berlin, 1981.

\bibitem{Caleiro}  Caleiro C., Gon\c{c}alves R., \textquotedblleft Equipollent logical systems", in J.-Y. Beziau (Ed.), \emph{Logica Universalis: Towards a General Theory of Logic}, 2nd edition, Birkh\"auser Verlag, Basel, 2007, pp. 97--109.

\bibitem {Carnielli}  Carnielli W.A., D'Ottaviano I.M.L., \textquotedblleft Translations between logical systems: A manifesto", \emph{Logique Et Analyse}, 157, 1997, pp. 67-81.

\bibitem {Cassequent}  Casta\~{n}o D., Cornejo J.M., Viglizzo I.D., \textquotedblleft Gentzen-style sequent calculus for semi-intuitionistic logic", \emph{Studia Logica}, 104, 6, 2016, pp. 1245–-1265.

\bibitem {CV}  Cornejo J.M., Viglizzo I.D., \textquotedblleft On some semi-intuitionistic logics", \emph{Studia Logica}, 103, 2015, pp. 303–-344.

\bibitem {Czelakowski}Czelakowski J., \textquotedblleft Equivalential logics I", \emph{Studia Logica}, 45, 1981, pp. 227-236.

\bibitem{ERS} Egr\'e P., Rossi L., Sprenger J., \textquotedblleft De Finettian logics of indicative conditionals. Part I: Trivalent semantics and validity", \emph{Journal of Philosophical Logic}, 50, 2021, pp. 187--213.

\bibitem{EGK} Estrada Gonzalez L., Klonowski M., \textquotedblleft An analysis of poly-connexivity in Boolean connexive logics", talk presented at Trends in Logic 21, Frontiers of Connexive Logics, Bochum, 6-8 December 2021.

\bibitem {Fichtner}Fichtner K., \textquotedblleft Eine Bermerkung \"{u}ber ber Mannigfaltigkeiten universeller Algebren mit Idealen", \emph{Monatsh. d. Deutsch. Akad. d. Wiss. (Berlin)}, 12, 1970, pp. 21–-25.

\bibitem {Font}Font J.M., \emph{Abstract Algebraic Logic: An Introductory
Textbook}, College Publications, London, 2016.

\bibitem {FGQ}Fried E., Gr\"{a}tzer G., Quackenbush R., \textquotedblleft Uniform congruence
schemes", \emph{Algebra Universalis}, 10, 1980, pp. 176-189.

\bibitem {Gyuris}Gyuris V., \emph{Variations of Algebraizability}, Ph.D. thesis, The University of Illinois at Chicago, 1999.

\bibitem {GU}Gumm H.P., Ursini A., \textquotedblleft Ideals in universal algebra", \emph{Algebra Universalis}, 19, 1984, pp. 45-54.

\bibitem {Hiz}Hiz H., \textquotedblleft A warning about translating axioms", \emph{American Mathematical Monthly}, 65, 1958, pp. 613-614.

\bibitem {Humbiz} Humberstone L., \textquotedblleft Choice of primitives: A note on axiomatizing intuitionistic logic", \emph{History and Philosophy of Logic}, 19, 1, 1998, pp. 31--40.

\bibitem {Humberstone} Humberstone L., \textquotedblleft Contra-classical logics", \emph{Australasian Journal of Philosophy}, 78, 4, 2000, pp. 438--474.

\bibitem{JM} Jarmuzek T., Malinowski J., \textquotedblleft Boolean connexive logics: Semantics and tableau approach", \emph{Logic and Logical Philosophy}, 28, 3, 2019, pp. 427–-448.

\bibitem {Kapstrong}Kapsner A., \textquotedblleft Strong connexivity", \emph{Thought}, 1, 2, 2012, pp. 141--145.

\bibitem {Kapsbook}Kapsner A., \emph{Logics and Falsifications: A New Perspective on Constructivist Semantics}, Springer, Berlin, 2014.

\bibitem {Kapsomo}Kapsner A., Omori H., \textquotedblleft Superconnexivity reconsidered", talk presented at Trends in Logic 21, Frontiers of Connexive Logics, Bochum, 6-8 December 2021.

\bibitem{KP}K\"ohler P., Pigozzi D., \textquotedblleft Varieties with
equationally definable principal congruences", \emph{Algebra Universalis}, 11, 1980, pp. 213--219.

\bibitem{MP}Mares E., Paoli F., \textquotedblleft C.I. Lewis, E.J. Nelson, and the modern origins of connexive logic", \emph{Organon F}, 26, 2019, pp. 405-–426.

\bibitem{McCall}McCall S., \textquotedblleft A history of connexivity”, in: D.M. Gabbay et al. (Eds.), \emph{Handbook of the History of Logic}, vol. 11, Elsevier, Amsterdam, 2012, pp. 415-–449.

\bibitem {NvP}Negri S., von Plato J., \emph{Structural Proof Theory}, Cambridge University Press, Cambridge, 2008.

\bibitem{Nelson}Nelson E.J., \textquotedblleft Intensional relations", \emph{Mind}, 39, 156, 1930, pp. 440–453.

\bibitem{OW}Omori H., Wansing H., \textquotedblleft Connexive logics. An overview and current trends", \emph{Logic and Logical Philosophy}, 28, 3, 2019, pp. 371–387.

\bibitem {Pynko}Pynko A., \textquotedblleft Definitional equivalence and algebraizability
of generalized logical systems", \emph{Annals of Pure and Applied Logic}, 98, 1999, pp. 1--68.

\bibitem{Raftery}Raftery J.G., \textquotedblleft Correspondences between Gentzen and Hilbert systems", \emph{Journal of Symbolic Logic}, 71, 3, 2006, pp. 903--957.

\bibitem{Sanka}Sankappanavar H.P., \textquotedblleft Semi-Heyting algebras: An abstraction from Heyting algebras", \emph{Actas del IX Congreso dr. Antonio A.R. Monteiro}, 2007, pp. 33--66.

\bibitem{Shapiro}Shapiro S., \textquotedblleft Incomplete translations of complete logics", \emph{Notre Dame Journal of Formal Logic}, 18, 2, 1977, pp. 248--250.

\bibitem{SV1}Spinks M., Veroff R., \textquotedblleft Constructive logic with strong negation is a substructural logic I", \emph{Studia Logica}, 88, 3, 2008, pp. 325-348.

\bibitem{SV2}Spinks M., Veroff R., \textquotedblleft Constructive logic with strong negation is a substructural logic II", \emph{Studia Logica}, 89, 3, 2008, pp. 401-425.

\bibitem {TvD}Troelstra A., van Dalen D., \emph{Constructivism in Mathematics}, 2 vols., North Holland, Amsterdam, 1988.

\bibitem{vA}van Alten C., \emph{An Algebraic Study of Residuated Ordered Monoids and Logics without Exchange and Contraction}, PhD Thesis, University of Natal, 1998.

\bibitem {Wan}Wansing A., \emph{The Logic of Information Structures}, Springer, Berlin, 1993.

\bibitem {Wanpts}Wansing H., \textquotedblleft The idea of a proof-theoretic semantics and the meaning of the logical operations", \emph{Studia Logica}, 64, 1, 2000, pp. 3--20.

\bibitem{Wanbochum} Wansing H., \textquotedblleft Connexive modal logic", In R. Schmidt et al. (Eds.), \emph{ Advances in Modal Logic}, King’s College Publications, London, 2005, pp. 367–383.

\bibitem{Wanstanford} Wansing H., \textquotedblleft Connexive logic", In E.N. Zalta (Ed.), \emph{The Stanford Encyclopedia of Philosophy (Spring 2021 Edition)}, https://plato.stanford.edu/archives/spr2021/entries/logic-connexive/

\bibitem {WOF}Wansing H., Omori H., Ferguson T.M., \textquotedblleft Editorial preface", \emph{IfCoLog}, special issue on connexive logics, 3, 3, 2016, pp. 279--295.

\end{thebibliography}
\end{document}